\RequirePackage[l2tabu, orthodox]{nag}
\RequirePackage{fixltx2e}
\documentclass[10pt,a4paper,reqno,oneside,final]{smfart}

\usepackage{leftidx}
\usepackage{appendix}
\usepackage{tikz}\usetikzlibrary{decorations.markings,matrix,arrows}
\usepackage{amsfonts}
\usepackage{amsthm}
\usepackage[T1]{fontenc}
\usepackage[mathscr]{eucal}
\usepackage{setspace}
\usepackage{amssymb,latexsym,amsmath,amscd}
\usepackage{graphicx}
\usepackage{showkeys}
\usepackage[french]{babel}
\usepackage{layout}
\usepackage{enumerate}
\usepackage{indentfirst}
\usepackage[varg]{pxfonts}
\usepackage[stretch=10]{microtype}
\usepackage{booktabs}
\usepackage{xspace}
\usepackage{eufrak}
\usepackage[colorlinks=false, pdfborder={0 0 0}]{hyperref} 
\usepackage{nameref}
\usepackage{cleveref}
\usepackage{calrsfs}
\usepackage{filecontents}
\usepackage{mathtools}
\usepackage{titlesec}
\usepackage{fixltx2e}
\usepackage{todonotes}

\titleformat{\paragraph}
{\normalfont\normalsize\bfseries}{\theparagraph}{1em}{}
\titlespacing*{\paragraph}
{0pt}{3.25ex plus 1ex minus .2ex}{1.5ex plus .2ex}

\newtheoremstyle{exostyle} 
{\topsep}% espace avant 
{\topsep}% espace apres 
{}% Police utilisee par le style de thm 
{}% Indentation (vide = aucune, \parindent = indentation paragraphe) 
{\bfseries}% Police du titre de thm 
{.}% Signe de ponctuation apres le titre du thm 
{ }% Espace apres le titre du thm (\newline = linebreak) 
{\thmname{#1}\thmnumber{ #2}\thmnote{. \normalfont{\textit{#3}}}}% composants du titre du thm : \thmname = nom du thm, \thmnumber = numéro du thm, \thmnote = sous-titre du thm 

\theoremstyle{exostyle} 
 
\makeatletter
\newcommand{\neutralize}[1]{\expandafter\let\csname c@#1\endcsname\count@}
\makeatother

\theoremstyle{plain}
\newtheorem{thm}{Theorem}[section]

\newenvironment{thmbis}[1]
  {\renewcommand{\thethm}{\ref{#1}$'$}%
   \neutralize{thm}\phantomsection
   \begin{thm}}
  {\end{thm}}
  
\newtheorem*{thm*}{Theorem}

\newtheorem{lem}[thm]{Lemma}

\newtheorem{pro}[thm]{Proposition}
\newtheorem{pro-def}[thm]{Proposition-Definition}
\newtheorem{cor}[thm]{Corollary}
\newtheorem{conj}[thm]{Conjecture}

\theoremstyle{definition}
\newtheorem{Def}[thm]{Definition}

\newtheorem{rem}[thm]{Remark}

\theoremstyle{remark}

\newcommand{\ssec}{\subsection}
\newcommand{\sssec}{\subsubsection}

\newcommand{\wt}{\widetilde}

\newcommand{\bP}{\mathbf{P}}

\newcommand{\bN}{\mathbf{N}}
\newcommand{\bC}{\mathbf{C}}

\newcommand{\bQ}{\mathbf{Q}}

\newcommand{\bZ}{\mathbf{Z}}

\newcommand{\gO}{\Omega}
\newcommand{\ga}{\alpha}
\newcommand{\gS}{\Sigma}
\newcommand{\gb}{\beta}

\newcommand{\gs}{\sigma}

\newcommand{\cF}{\mathcal{F}}

\newcommand{\fA}{\mathfrak{A}}

\newcommand{\fS}{\mathfrak{S}}

\newcommand{\colonec}{\mathrel{:=}}

\newcommand{\pr}{\mathrm{pr}}

\newcommand{\CH}{\mathrm{CH}}

\newcommand{\Pic}{\mathrm{Pic}}

\newcommand{\rat}{\mathrm{rat}}

\newcommand{\red}{\mathrm{red}}

\newcommand{\Ima}{\mathrm{Im}}

\newcommand{\ul}{\underline}

\renewcommand{\(}{\left(}
\renewcommand{\)}{\right)}

\newcommand{\vast}{\bBigg@{4}}
\newcommand{\Vast}{\bBigg@{5}}

\newcommand*\eto{%
  \xrightarrow[]{\raisebox{-0.25 em}{\smash{\ensuremath{\sim}}}}%
}

\makeatletter
\let\orgdescriptionlabel\descriptionlabel
\renewcommand*{\descriptionlabel}[1]{%
  \let\orglabel\label
  \let\label\@gobble
  \phantomsection
  \edef\@currentlabel{#1}%
  \let\label\orglabel
  \orgdescriptionlabel{#1}%
}
\makeatother
\tikzset{node distance=2cm, auto}

\numberwithin{equation}{section}

\title{On the Chow group of zero-cycles of a generalized Kummer variety} 

\author{Hsueh-Yung Lin}
\begin{document}

\maketitle

\begin{abstract}
For a generalized Kummer variety $X$ of dimension $2n$, we will construct for each $0 \le i \le n$ some co-isotropic subvarieties in $X$ foliated by  $i$-dimensional constant cycle subvarieties. These subvarieties serve to prove that the rational orbit filtration introduced by Voisin on the Chow group of zero-cycles of a generalized Kummer variety coincides with the induced Beauville decomposition from the Chow ring of abelian varieties. As a consequence, the rational orbit filtration is opposite to the conjectural Bloch-Beilinson filtration for generalized Kummer varieties.

\end{abstract}

\section{Introduction}

The motivation of this work comes from the study of the Chow group of zero-cycles of a hyper-Kähler manifold. Based on the existence of the Beauville-Voisin zero-cycle in a projective $K3$ surface~\cite{BV}, which up to a scalar multiple is the intersection of any two divisor classes, Beauville asked in~\cite{BeauvilleSplitBB} for a projective hyper-Kähler manifold $X$, whether the Bloch-Beilinson filtration $F^\bullet_{BB}$ on the Chow ring of $X$ with \emph{rational} coefficients $\CH^\bullet(X)$, if exists, admits a multiplicative splitting  as in the case of abelian varieties~\cite{BeauvilleDecAb}. While the existence of the Bloch-Beilinson filtration is still largely conjectural, Beauville also formulated his weak splitting conjecture~\cite{BeauvilleSplitBB} which is predicted by the existence and the splitting of $F^\bullet_{BB}$, but does not rely on the existence of $F^\bullet_{BB}$. The reader is referred to~\cite{BeauvilleSplitBB,VoisinBVconjHilbK3,FuBVKummer,RiessBconjLF} for the partial results of this conjecture.

Another way to approach the splitting question, at least for $\CH_0(X)$,  is based on the following observation due to Voisin~\cite[Lemma $2.2$]{Voisin0cycleK3}. If $C$ is a curve in a projective $K3$ surface $S$ such that all points in $C$ are rationally equivalent \emph{in $S$}, then the Beauville-Voisin zero-cycle is the class of any point in $C$. Such a curve and its higher dimensional analogue in a hyper-Kähler manifold led Huybrechts to introduce the following definition: 
\begin{Def}[\cite{HuybCCC}]
A \emph{constant cycle subvariety} $Y \subset X$ is a subvariety such that all points in $Y$ are rationally equivalent in $X$. For $0 \le i \le n$, a \emph{$C^i$-subvariety} $Y$ of $X$ is a subvariety of dimension $2n-i$ such that $\dim O_x \ge i$ for every $x \in Y$.
\end{Def}

Constant cycle subvarieties are used by Voisin in~\cite{VoisinHKcoisotp} to introduce the \emph{rational orbit filtration} $S_\bullet\CH_0(X)$ as follows. 

\begin{Def}[\cite{VoisinHKcoisotp}]
 For any integer $p$, the subgroup  $S_p\CH_0(X)$ is generated by the classes of points $x \in X$ supported on a constant cycle subvariety of dimension $p$.
 \end{Def}

  \emph{A priori} this filtration is indexed by $0 \le p \le 2n \colonec \dim X$, but since constant cycle subvarieties are isotropic with respect to any holomorphic symplectic two-form on $X$,  by Mumford-Roitman's theorem~\cite[Proposition $10.24$]{VoisinII}, $S_p\CH_0(X)$ vanishes for $p > n$. As for $F^\bullet_{BB}\CH_0(X)$,  since $H^0(X,\gO_X^{2p-1})$ vanishes for all integer $p$,  from the axioms of $F^\bullet_{BB}$~\cite[Chapter $11$]{VoisinII} we see that $F^{2p-1}_{BB}\CH_0(X) = F^{2p}_{BB}\CH_0(X)$. The following conjecture was formulated in~\cite{VoisinHKcoisotp}

\begin{conj}[\cite{VoisinHKcoisotp}]\label{conj-opp}
The rational orbit filtration $S_{\bullet}CH_0(X)$ is opposite to the Bloch-Beilinson filtration $F_{BB}^\bullet \CH_0(X)$. Equivalently,  the restriction to $S_p\CH_0(X)$ of the quotient map $\CH_0(X) \to \CH_0(X)/ F^{2n - 2p + 1}_{BB}\CH_0(X)$ is  an isomorphism.

\end{conj}

 This conjecture has been verified by Voisin~\cite{VoisinHKcoisotp} for a punctual Hilbert scheme of a $K3$ surface and a Fano variety of lines on a cubic fourfold, for which we know explicitly some candidate for the Bloch-Beilinson filtration~\cite{ShenVialFM}. We refer to~\cite{VoisinHKcoisotp} for details and further discussions of this conjecture.

%Let $X$ be a generalized Kummer variety defined by an abelian surface $A$ and  $\CH_0(X) \colonec \CH_0(X)_\bQ$ its Chow group of zero-cycles with rational coefficients. We can define two decreasing filtrations on $\CH_0(X)$.

Motivated by the  approach described above to Beauville's splitting conjecture, the goal of this article is to study two natural filtrations defined on the Chow group of zero-cycles $\CH_0(X)$ for a generalized Kummer variety $X$. The first one is the rational orbit filtration $S_\bullet\CH_0(X)$ introduced above, and  the second filtration comes from the abelian nature of a generalized Kummer variety. Let $A$ be an abelian surface and $X$ the generalized Kummer variety of dimension $2n$ defined by $A$. If $A_0^{n+1}$ denotes the kernel of the sum map $\mu: A^{n+1} \to A$, then $X$ is a desingularization of $A_0^{n+1} / \fS_{n+1}$ where $\fS_{n+1}$ acts as permutations of factors. So the quotient map $A_0^{n+1}  \to A_0^{n+1} / \fS_{n+1}$ induces an isomorphism between $\CH_0(X)$ and $\CH_0(A_0^{n+1})^{\fS_{n+1}}$, where $\CH_0(A_0^{n+1})^{\fS_{n+1}}$ is the $\fS_{n+1}$-invariant part of $\CH_0(A_0^{n+1})$. As the $\fS_{n+1}$-action on $\CH_0(A_0^{n+1})$ is compatible with the Beauville decomposition $\bigoplus_{0\le s \le 2n}\CH_0(A_0^{n+1})_s$, the last decomposition induces a decomposition $\bigoplus_{0\le s \le 2n}\CH_0(X)_s$ of $\CH_0(X)$, which defines the second filtration.

%The study of these two filtrations is mainly motivated by Beauville's conjecture on the splitting of the conjectural Bloch-Beilinson filtration for $X$. On one hand, provided the existence of the Bloch-Beilinson filtration $F^\bullet_{BB}\CH(Y)$ for all smooth projective varieties $Y$, the rational orbit filtration $S_\bullet\CH_0(V)$ for a projective hyper-Kähler manifold $V$ is conjectured in~\cite{VoisinHKcoisotp} to be opposite to $F^\bullet_{BB}\CH_0(V)$ in the following sense:   since $H^0(V,\gO_V^{2p-1})$ vanishes for all odd numbers $p$,  from the axioms of $F^\bullet_{BB}$~\cite[Chapter $11$]{VoisinII} we see that $F^{2p-1}_{BB}\CH_0(V) = F^{2p}_{BB}\CH_0(V)$, and the restriction to $S_p\CH_0(V)$ of the quotient map $\CH_0(V) \to \CH_0(V)/ F^{2n - 2p + 1}_{BB}\CH_0(V)$ is conjectured to be an isomorphism. This conjecture has been verified by Voisin~\cite{VoisinHKcoisotp} for a punctual Hilbert scheme of a $K3$ surface and a Fano variety of lines on a cubic fourfold, for which we know explicitly some candidate of the Bloch-Beilinson filtration. The reader is referred to~\cite{VoisinHKcoisotp} for details and further discussion of this conjecture. 
 
 Contrary to the case of $S_p\CH_0(X)$ where the splitting property is conjectural, since the Beauville decomposition of $\CH_0(A_0^{n+1})$  splits  the Bloch-Beilinson filtration on  $\CH_0(A_0^{n+1})$~\cite[\S 2.5]{MurreFiltrBBM}, by functoriality the aforementioned induced Beauville decomposition $\bigoplus_{ 0 \le s \le 2n }\CH_0(X)_s$ actually defines a natural splitting of $F^\bullet_{BB}\CH_0(X)$. As before, the axioms that $F^\bullet_{BB}$ should satisfy implies the vanishing of $\CH_0(X)_s$ for any odd number $s$ since $H^0(X,\gO_X^s) = 0$. In fact, this vanishing property can be proven unconditionally (\emph{cf.} Sections~\ref{sec-Beauvillefiltr} and~\ref{sec-coin} for two different proofs).
 
\begin{thm}\label{thm-van}
For any odd number $s$, $\CH_0(X)_s = 0$.
\end{thm}

The main result of this paper is to show that on $\CH_0(X)$, the rational orbit filtration coincides with the induced Beauville decomposition.

%defined by an abelian surface $A$, the rational orbit filtration on the Chow group of zero-cycles with rational coefficients $\CH_0(X) \colonec \CH_0(X)_\bQ$ coincides with the induced Beauville decomposition from the Chow ring of abelian varieties. Let us first introduce the two objects that we will compare, for which the reader is referred to Section~\ref{sec-Beauvillefiltr} for more precise definitions. 
%Secondly,  
%We will show that $\CH_0(X)_{2p+1} = 0$ for all $p$ and the third filtration is defined by $B_p\CH_0(X) \colonec \bigoplus_{0\le s \le 2n - 2p}\CH_0(X)_s$.

\begin{thm}\label{thm-coin}
If $X$ is a generalized Kummer variety of dimension $2n$, then 
$$S_p\CH_0(X)  =  \bigoplus_{2s \le 2n-2p}\CH_0(X)_{2s}.$$
\end{thm} 

In particular, Theorem~\ref{thm-coin} implies that the rational orbit filtration $S_\bullet \CH_0(X)$ is opposite to the Bloch-Beilinson filtration $F^\bullet_{BB}\CH_0(X)$ (assumed to exist),  thus proving Conjecture~\ref{conj-opp} in the case of generalized Kummer varieties.

%Since the induced Beauville filtration defines a splitting on $F^\bullet_{BB}\CH_0(X)$
%the two filtrations opposite to $F^\bullet_{BB}\CH_0(X)$ are the same.

%Note that $S_\bullet\CH_0(X)$ is defined for any variety $X$, especially when $X$ is an irreducible holomorphic symplectic variety,  one also has $S_p\CH_0(X) = 0$ for $p > n$. When $X$ is a punctual Hilbert schemes of a $K3$ surface $S$, the filtrations $S_\bullet\CH_0(X)$ and a version of $N_\bullet\CH_0(X)$ where $0 \in A$ is replaced by any point in $S$ representing the Beauville-Voisin zero-cycle was first introduced and studied by C. Voisin. Precisely, she shows that $S_\bullet\CH_0(X)$ is opposite to the Bloch-Beilinson filtration, thus admits a splitting given by Künneth projectors.

%In this article, we will show that the same consequences hold for generalized Kummer varieties:
The outline of the proof of Theorem~\ref{thm-coin} is as follows. Let $A^{n+1}_0 \subset A^{n+1}$ be the kernel of the sum map $\mu_{n+1} : A^{n+1} \to A$ and let $K_{(n)} \subset A^{(n+1)}$ be the image of $A^{n+1}_0$ under the quotient map $q_{n+1} : A^{n+1} \to A^{(n+1)}$. Thus $K_{(n)}$ can be defined as 
\begin{equation}\label{eqn-Kn}
A_0^{n+1}/ \fS_{n+1},
\end{equation}
where $\fS_{n+1}$ is the symmetric group which acts naturally on $A_0^{n+1}$.
Since the Hilbert-Chow morphism $\nu : X \to K_{(n)} $ induces an isomorphism $\CH_0(X) \eto \CH_0(K_{(n)})$ (where the Chow groups are defined with rational coefficients), we will be working with $K_{(n)}$ instead of $X$. Suppose that $X$ is defined by the jacobian of a smooth curve of genus two $C$. Using the symmetric products of  $C$ and observing that the abelian sum $C^{(i)} \to A \colonec J(C)$ is generically a $\bP^{i-2}$-fibration, we will first construct for all $0 \le i\le n$ and $k \in \bZ_{>0}$, subvarieties $V_{i,k} \subset K_{(n)}$ of dimension $2n-i$, subject to the following property:

\begin{pro}[ = Corollary~\ref{cor-main}]\label{pro-main}
$V_{i,k}$ is  swept out by constant cycle subvarieties of dimension $i$ and for any constant cycle subvariety $Y \subset K_{n}$ of dimension $i$, there exist $k \in \bZ_{>0}$ and  $ p \in V_{i,k}$ such that $p$ is rationally equivalent to any point in $\nu(Y)$.
 In other words,   $\nu_*S_i\CH_0(K_{n})$ is supported on  $ \cup_{k>0}V_{i,k}$.
\end{pro}  

Next, we will prove that 
\begin{pro}[ = Lemma~\ref{lem-inclu} + Proposition~\ref{pro-suppBD}]\label{pro-second}
The rational equivalence class of a zero-cycle in $\cup_{k>0}V_{i,k}$ lies in $\CH_0(K_{(n)})_{\le 2n-2i} \colonec \oplus_{s \le 2n -2i} \CH_0(K_{(n)})_s$. Conversely,  $\CH_0(K_{(n)})_{\le 2n-2i}$ is supported on $V_{i,1}$.
\end{pro}
Combining Proposition~\ref{pro-main} and~\ref{pro-second}, Theorem~\ref{thm-coin} follows easily. 

The paper is organized as follows. We will first construct subvarieties $V_{i,k} \subset K_{(n)}$ in Section~\ref{sec-ccs} and then prove Proposition~\ref{pro-main} in Section~\ref{sec-supp}. The proof of Theorem~\ref{thm-van} can be found in Section~\ref{sec-Beauvillefiltr}, after introducing the induced Beauville decomposition. Finally in Section~\ref{sec-coin}, we will prove Proposition~\ref{pro-second} hence Theorem~\ref{thm-coin}.

% For any variety $X$ and $x \in X$, denoted by $O_x \subset X$ the set of points $x' \in X$ which are rationally equivalent to $x$. The set $O_x$ is a countable union of Zariski closed subsets of $X$; the dimension of $O_x$ is defined to be the maximal dimension of all its irreducible components.

%\begin{Def}
%For $i \ge 0$, we say that a subvariety $V$ of $X$ is a \emph{$C^i$-subvariety} if $\dim O_x \ge i$ for all $x \in V$.
%\end{Def} 

%Let $X$ be a variety. For any $x \in X$, denoted by $O_x \subset X$ the set of points $x' \in X$ which are rationally equivalent to $x$. The set $O_x$ is a countable union of Zariski closed subsets of $X$; the dimension of $O_x$ is defined to be the maximal dimension of all its irreducible components.

 %Let $X$ be an irreducible holomorphic symplectic variety of dimension $2n$. Note that by Mumford-Roitman's theorem~\cite[Proposition $10.24$]{VoisinII}, $2n-i$ is the maximum dimension of subvarieties of $X$ satisfying property $(C^i)$. In particular, there is no $C^i$-subvariety of $X$ for $i > n$.

\ssec{Conventions and notations}\label{convention} \hfill

In this paper, all varieties are defined over the field of complex numbers $\bC$. The Chow rings appearing in this paper are with \emph{rational} coefficients. 

For $n \in \bZ_{>0}$, the $n$-th symmetric product $X^n / \fS_n$ is denoted by $X^{(n)}$. We write $q_n : X^n \to X^{(n)}$ the quotient map. The Hilbert scheme of points  of length $n$ on $X$ is denoted by $X^{[n]}$ and $\nu_n : X^{[n]} \to X^{(n)}$ (or simply $\nu$ if there is no ambiguity) the associated Hilbert-Chow morphism.

We use additive notation for the group operation in an abelian variety $A$. Each integer $n \in \bZ$ defines a multiplication-by-$n$ map $[n] : A \to A$, whose kernel is denoted by $A[n]$. Since we will use the same notation for the addition of algebraic cycles in $A$, when a subvariety $V$ in $A$ is considered as an algebraic cycle, it will be systematically denoted by $\{V\}$ in order to avoid any confusion. 
%For $V_1,V_2 \in \CH^\bullet(A)$, we write $V_1 * V_2$ for their Pontryagin product. 
%Let $V_1,\ldots,V_n$ be subvarieties of $X$, we designate by $V_1 * \cdots * V_n$ the image of the natural composition map
%$$V_1 \times \cdots \times V_n \hookrightarrow X^n \to X^{(n)}.$$
%If $V_1 = \cdots = V_n = V$, this image is denoted by $V^{*n}$ for short.

%Let $A$ be an abelian surface. If $V$ is a subvariety of $A^n$, the translation of $V$ by an element $a \in A$ is denoted by
%$$\tau_a(V) \colonec a + V \colonec \left\{(x_1 + a, \ldots, x_n + a) \mid (x_1 , \ldots,x_n) \in V \right\}.$$
%The subvariety $\tau_a(V) = a+ V$ is defined similarly for $V \subset A^{(n)}$ or $V \subset A^{[n]}$.

\section{Constructing $C^i$-subvarieties in generalized Kummer varieties}\label{sec-ccs}

\ssec{Definitions}\hfill

Let $X$ be a variety. For any $x \in X$, let $O_x$ denote the set of points in $X$ which are rationally equivalent to $x$. Since $O_x$ is a countable union of Zariski closed subsets in $X$, we can define $\dim O_x$ to be the maximum of the dimension of its irreducible components.

\begin{Def}[\cite{VoisinHKcoisotp}]
A  \emph{$C^i$-subvariety} $Y$ of $X$ is a subvariety of dimension $2n-i$ such that $\dim O_x \ge i$ for every $x \in Y$.
\end{Def}

When $X$ is an algebraic hyper-Kähler manifold of dimension $2n$, for instance a generalized Kummer variety, a constant cycle subvariety is isotropic with respect to any holomorphic symplectic two-form on $X$  by Mumford-Roitman's theorem, so its dimension is at most $n$. As another application of Mumford-Roitman's theorem, any subvariety in $X$ covered by constant cycle subvarieties of dimension $i$ is of dimension at most $2n-i$~\cite[Theorem $1.3$]{VoisinHKcoisotp}. It is also proved in~\cite{VoisinHKcoisotp} that any $C^i$-subvariety is swept out by constant cycle subvarieties of dimension $i$. 
 In other words, $C^i$-subvarieties are exactly the subvarieties of maximal dimension sharing this property. 

The \emph{rational orbit filtration} $S_i\CH_0(X)$, defined for $i \in \bZ$, is the subgroup of $\CH_0(X) $ generated by the classes of points supported on some constant cycle subvariety of dimension $ \ge i$. The goal of Section~\ref{sec-ccs} is to construct some $C^i$-subvarieties in an (algebraic) generalized Kummer variety $X$ which support zero-cycles in $S_i\CH_0(X)$. Below we recall their definition and set up some conventions. 

Let $A$ be an abelian surface. For each $n \in \bN$,  let $\mu_{n+1} : A^{[n+1]} \to A$ denote the sum map. We will use the same notation $\mu_{n}$ to denote other sum maps like $ A^{(n)} \to A$ and $ A^{n} \to A$.  A \emph{generalized Kummer variety} is defined to be one of the fibers of the iso-trivial fibration $\mu : A^{[n+1]} \to A$ and  is denoted by $K_n(A)$ or $K_n$ if there is no ambiguity. 

\begin{rem}\label{rem-quotsing}
 In general, if $f : X \to Y$ is a morphism between quotient varieties of non-singular varieties by some finite group action, then  by~\cite[Example $16.1.13$]{Fulton} the pullback map $f^* : \CH^\bullet(Y) \to \CH^\bullet(X)$ (where we recall that the Chow groups are defined with rational coefficients) is well-defined. If $f$ is birational, then $f^* : \CH_0(Y) \to \CH_0(X)$ is an isomorphism. In particular, applying this to the Hilbert-Chow morphism $\nu : K_n \to K_{(n)}$, we conclude that  $\nu_* \CH_0(K_n) \to \CH_0(K_{(n)})$ is an isomorphism.
 
  It also follows that if  $Z$ is a $C^i$-subvariety in $Y$, then the proper transformation of $Z$ under $f^{-1}$  is a $C^i$-subvariety in $X$. We see that the $C^i$-subvarieties $V_{i,k}$ constructed above lift to $C^i$-subvarieties in $K_n$.
\end{rem}

The following result will be useful.

\begin{lem}\label{lem-isog}
If $f :A' \to  A$ is an isogeny, then $f$ induces an isomorphism of Chow groups $\CH_0(K_n(A')) \simeq \CH_0(K_n(A))$.
\end{lem} 

\begin{proof}
First using Remark~\ref{rem-quotsing}, it suffices to prove that the natural morphism $\CH_0(K_{(n)}(A')) \to \CH_0(K_{(n)}(A))$ is an isomorphism. Using formula~(\ref{eqn-Kn}), this last fact follows from the fact that the morphism $\CH_0({A'_0}^{n+1}) \to \CH_0(A_0^{n+1})$ is an isomorphism, since ${A'_0}^{n+1} \to A_0^{n+1}$ is an isogeny of abelian varieties~\cite{Blochab}.  
\end{proof}

Thanks to Lemma~\ref{lem-isog} we can suppose that $A$ is a principally polarized abelian surface $(A,C)$. So either $(A,C)$ is the Jacobian variety $J(C)$ of a genus $2$ curve $C$ together with an Abel-Jacobi embedding $C \hookrightarrow A$ defining the theta divisor, or is the product of two elliptic curves $E \times E'$ with $C = E \times \{o'\} \cup \{o\} \times E'$ where $o \in E$ and $o' \in E'$ are the origine of $E$ and $E'$.  We assume that the origine of $A$ is a Weierstrass point of $C$  in the former case,, and $(o,o')$ in the latter case.

\ssec{Construction} \hfill

%\begin{lem}\label{lem-SVCCbirat}
%Suppose $f$ is birational and   $f^* : \CH_0(Y) \to \CH_0(X)$ is an isomorphism. If $Z$ is a $C^i$-subvariety in $Y$, then $f^{-1}(Z)$  is a $C^i$-subvariety in $X$. \qed
%\end{lem}

%Since $K_{(n)} \colonec \nu(K_n)$ is the quotient of the kernel of the sum map $\mu : A^{n+1} \to A$ by the action of $\fS_{n+1}$ permuting different components in $A^{n+1}$, the pullback between Chow groups (with rational coefficients) under the Hilbert-Chow morphism $\nu : K_n \to K_{(n)}$ is well defined. The preceding discussion shows that all $C^i$-subvarieties  in $K_{(n)}$ lift to some $C^i$-subvarieties in $K_n$.

%Let $\mu_{n+1} : A^{n+1} \to A$ be the sum map and $[n+1] : A \to A$ the multiplication by $n+1$. On one hand, the base change of $\mu_{n+1}$ by $[n + 1]$ is isomorphic to the second projection $\wt{K_n} \times A \to A$ where 
%$$\wt{K_n} \colonec \left\{(a_0,\ldots,a_n) \in A^{n+1} \ \big{|} \ \sum_{i=0}^n a_i = 0\right\}.$$
%On the other hand, the base change of $[n+1]$ under $s_{n + 1}$ factorizes through 

Now we construct for all $0 \le i \le n$ and $k \in \bZ_{>0}$, a $C^i$-subvariety $V_{i,k}$  in $K_{(n)}$ where $X$ is the generalized Kummer variety defined by a principally polarized abelian surface $(A,C)$. These $V_{i,k}$'s will be used in Section~\ref{sec-Beauvillefiltr} to prove that  $S_\bullet\CH_0(X)$ and the induced Beauville filtration are the same. 

Set  
\begin{equation}\label{def-tau}
\begin{split}
\tau_{n+1} :  A^{n+2} & \to A^{n+1} \\
 (\ul{a},a) \colonec \(a_0,\ldots,a_n, a \) & \mapsto \tau_a(\ul{a}) \colonec \(a_0 + a,\ldots ,a_{n} + a \);
\end{split}
\end{equation}
We will omit the index $n+1$ when there is no ambiguity. For $Z_1 \subset A^{n+1}$ and $Z_2 \subset A$, we also define
\begin{equation}\label{def-tau1}
\tau(Z_1,Z_2) \colonec \tau(Z_1 \times Z_2). 
\end{equation}

For $k \in \bZ_{>0}$, let $C_k $ be the pre-image of the multiplication-by-$k$ map $[k] : A \to A$ of the theta divisor $C \subset A$; in the case where $A = E \times E'$, $C_k$ is the union of all $\tau_{(a,a')}(C) = E \times \{a'\} \cup \{a\} \times E'$ as $a \in E$ and $a' \in E$ run through all $k$-torsion points.

 For $0 \le i \le n$, let 
 $$E_{i,k} \colonec C_k^{i+2} \times A^{n-i} \subset A^{n+1} \times A,$$
and set
%Under the Hilbert-Chow morphism, the locus of reduced schemes $A^{[n+1]}_{\red}$ is isomorphic onto its image to the (open) maximal multiplicity-stratum of $A^{(n+1)}$ (consisting of zero-cycles whose support has cardinality $n+1$) and will be denoted by $A^{(n+1)}_{\red}$. Set $K_n^\circ(A) \colonec K_n(A) \cap A^{[n+1]}_{\red}$.
$$V_{i,k} \colonec q_{n+1}\( \tau(E_{i,k})\) \cap K_{(n)}$$ 
where  we recall that $q_{n+1} : A^{n+1} \to A^{(n+1)}$ is the quotient map.

%For $k \in \bZ_{>0}$, define the morphisms
%\begin{equation}\label{eq-orbitmap}
%\begin{split}
%\tau_{i,k} : A \times C^{(i+2)} \times A^{(n-i-1)} & \to A^{(n+1)} \\
 %\(a, c_0 * \cdots * c_{i+1}, a_{i+2} * \cdots * a_{n} \) & \mapsto   \{C_k_0 + a\} * \cdots * \{ C_k_{i+1} + a \} *  \{a_{i+2} \} * \cdots * \{ a_{n} \}
%\end{split}
%\end{equation}
%for $0 \le i < n$, and
%\begin{equation}\label{eq-orbitma\hat{p}}zero-
%\begin{split}
%\tau_{n,k} : C \times C^{(n+1)} & \to A^{(n+1)} \\
 %\(c, c_0 * \cdots * c_{n} \) & \mapsto   \{C_k_0 + C_k\} * \cdots * \{ C_k_{n} + C_k \}
%\end{split}
%\end{equation}
%Then it is easy to see that $V_{i,k} = \nu^{-1}\(\Ima(\tau_{i,k})\) \cap K_n$. 

\begin{lem}~\label{lem-critsum}
The rational equivalence class of a point  $z =  \sum_{j= 1}^{i+2} \{a + c_j\} +\sum_{j= i+3}^{n+1} \{a_j\}$ in $K_{(n)}$ is independent of $c_1,\ldots,c_{i+2} \in C_k$ whenever $\sum_j c_j$ is fixed in $A$. Similarly,  the rational equivalence class of $z$ \emph{as a zero-cycle in $A$} is also independent of $c_1,\ldots,c_{i+2} \in C_k$ whenever $\sum_j c_j$ is fixed.
\end{lem}
\begin{proof}
%\todo{réécrire la preuve en utilisant l'isogénie $[k] : A \to A$ et le fait que pour $k=1$, les fibres de $C^{(l)} \to A$ sont $\CH_0$-triviales}
The fibers of the sum map $\mu_2 : C^{(2)} \to A$ are $\CH_0$-trivial varieties. When $C$ is smooth, recall that for any $l > 2$, the abelian sum map  $\mu_l : C^{(l)} \to A$ is a $\bP^{l-2}$-fibration. So if $(A,C)$ is any principally polarized abelian surface, the fibers of $\mu_l : C^{(l)} \to A$ are also $\CH_0$-trivial varieties since $\mu_l $ is a specialization of a family of $\bP^{l-2}$-fibrations. 
 
Now let $k \in \bZ_{>0}$. Since an isogeny $B \to B'$ between abelian varieties induces a natural isomorphism $\CH_0(B) \simeq \CH_0(B')$~\cite{Blochab}, and since the push-forward of a zero-cycle in $A^{(l)}$ supported on a fiber $F_k$ of the sum map $C_k^{(l)} \to A$ under the isogeny $[k] : A \to A$ is supported on $[k](F_k)$, which is a fiber of $\mu_l : C^{(l)} \to A$ so constant cycle in $A^{(l)}$, we conclude that $F_k$ is a constant cycle subvariety in $A^{(l)}$. Thus if $l = i+2$, then the image of $F_k$ under the map $A^{(l)} \to K_{(n)}$ sending $z$ to $z +\sum_{j= i+3}^{n+1} \{a_j\}$ is a constant cycle subvariety, which proves the first assertion. Since the push-forward of points in $F_k$ under the incidence correspondence $\CH_0(K_n) \to \CH_0(A)$ has constant rational equivalence class,  the second assertion follows. 

\end{proof}

\begin{pro} \label{pro-Ci}
%\begin{enumerate}[i)]
$V_{i,k}$ is a $C^i$-subvariety of dimension $2n-i$ in $K_{(n)}$.
% such that $V_{i,k} \cap A^{(n+1)}_\red \ne \emptyset$. 
%\item For each $z \in V_{i,k}$, there exiszero-t $a_{i+1},\ldots,a_{n+1} \in A$ such that $z$ is rationally equivalent to $i\cdot \{0\}  + \{a_{i+1}\} + \cdots \{a_{n+1}\}$ in $K_{(n)}$.
%\end{enumerate}
\end{pro}

\begin{proof}

Fix $k  \in \bZ_{>0}$. %The non-emptiness of $V_{i,k} \cap A^{(n+1)}_\red$ is obvious.
Since the sum map $\mu_2 : C^{(2)} \to A$ is birational, it is easy to see that  $\tau_{|E_{i,k}}$ is generically finite. So $\dim V_{i,k}  \ge  2n-i$.

%Similarly when $i = n$, $\dim V_{n,k} \ge n$. 
%That $\dim V_i = 2n-i$ follows from the surjectivity of the Abel-Jacobi map $C^{[2]} \to A$. 

 %Hence the fiber of the sum map $C^{(g)}_{p/q}\to A$ over $s \in A$ is
%$$\bigcup_{z \in A[p]} [p]\(F_z\),  \text{ where } \  \ \ F_z \colonec \left\{ \sum_j \{c_j\} \in C^{(g)}_{1/q} \  \big{|} \  \sum_j c_j = s + z \right\} \simeq \bP^{g-2}.$$ 
%As $[p]$ is a finite morphism, we see that the fibers of $C^{(g)}_{p/q}\to A$ are a union of $\CH_0$-trivial varieties of dimension $g-2$.

%It suffices to show that $F_0 \cap F_z = \emptyset$ for all $z\in A[p]$.

%\begin{lem}
%The fibers of $C^{(g)}_{p/q}\to A$ are $\CH_0$-trivial varieties. 

%\end{lem}

%\begin{proof} 
%\end{proof}

  When $i < n$, note that $V_{i,k}$ is covered by subvarieties
  $$F_b \colonec \left\{ \sum_{j=0}^{i+1}\{c_j+a\}  +  \sum_{j=i+2}^{n}\{a_{j}+a\} \in K_{(n)}\  \bigg{|} \  c_0,\ldots,c_{i+1} \in C,\ a \in A,\ \sum_{j=0}^{i+1}\{c_j\}=b, \right\}$$
  for all $b \in A$, which are constant cycle subvarieties of dimension $i$ by Lemma~\ref{lem-critsum}. We conclude by~\cite[Theorem $1.3$]{VoisinHKcoisotp} that $\dim V_{i,k} = 2n-i$, so $V_{i,k}$ is a $C^i$-subvariety.
  %define 
%\begin{equation}\label{eq-orbitmap}
%\begin{split}
%\tau_i :  C_k^{(i+2)} \times \(A^{(n-i-1)}\times A\) & \to A^{(n+1)} \\
 %\(\sum_{j=0}^{i+1}\{c_j\} , \sum_{j=i+2}^{n}\{a_{j}\} , a \) & \mapsto    \sum_{j=0}^{i+1}\{c_j+a\}  +  \sum_{j=i+2}^{n}\{a_{j}+a\}.
%\end{split}
%\end{equation}
%Then we have $q_{n+1}\(\tau(E_{i,k})\) = \Ima(\tau_i)$. Let 
%$$\pi_1 : C_k^{(i+2)} \times \(A^{(n-i-1)}\times A\) \to C_k^{(i+2)} \ \text{ and } \ \pi_2 : C_k^{(i+2)} \times \(A^{(n-i-1)}\times A\) \to   A^{(n-i-1)}\times A$$ be the first and the second projection.
%Since the image of the restriction of $\tau_i$ to a fiber of  $\(\mu \circ \pi_1, \pi_2\) : C_k^{(i+2)} \times A^{(n-i-1)} \times A \to A \times  A^{(n-i-1)} \times A$ is just $\CH_0$-trivial, and  is contained in some fiber of $\mu : A^{(n+1)} \to A$, we see that $V_{i,k}$ is covered by $\CH_0$-trivial varieties.  Hence  $V_{i,k}$ is a $C^i$-subvariety of dimension $2n-i$. 

%To see $ii)$, let $\ul{c} = \{c_0\} + \cdots + \{c_{i+1}\} \in C^{(i+2)}_{1/q}$, $\ul{a} = \{a_{i+2} \} + \cdots + \{a_{n}\}$, and $a \in A$. By the Jacobi inversion theorem, there exist $c,c' \in C_{1/q}$ such that $\sum_j c_j = c + c'$. From the description above of the fibers of $C^{(g)}_{p/q}\to A$, we see that $\{\tau_a\}_* \([p]_*\ul{c} + \ul{a}\)$ and   $\{\tau_a\}_* \([p]_* \(i\cdot \{0\}  + \{c\} + \{c'\}\) + \ul{a}\) \in K_{(n)}$ lie in the same image of $\bP^i$, which proves $ii)$.

%the restriction of $\tau_i$ to any fiber of the standard projection $ C_k^{(i+2)} \times A^{(n-i-1)}\times A \to A^{(n-i-1)}\times A$ is a finite morphism

In the case $i = n$, let $z = \sum_{j = 0}^n \{c_j + c\} \in V_{n,k}$ where $c,c_0,\ldots,c_n \in C_k$. Since
$$\sum_{j=0}^n c_j =  (n+1)\cdot (-c),$$
$z$ is rationally equivalent to $ (n+1) \cdot \{0\}$ in $K_{(n)}$ by Lemma~\ref{lem-critsum}. Hence $V_{n,k}$ is a constant cycle subvariety of dimension $n$.

\end{proof}

We terminate this section by the following result which is a direct consequence of Lemma~\ref{lem-critsum}. This gives simple representatives of  classes of points supported on $V_{i,k}$ modulo rational equivalence in $K_{(n)}$. 

\begin{lem}~\label{lem-rep}
If $i<n$, every $z \in V_{i,k}$ is rationally equivalent in $K_{(n)}$ to
%\begin{enumerate}[i)]
%\item $i\cdot \{a\} + \sum_{j= i+1}^{n+1} \{a_j\}$ for some $a,a_{i+1},\ldots,a_{n+1} \in A$ such that $i \ a + \sum_{j= i+1}^{n+1} a_j = 0$
%\item $(i+1) \cdot \{a+c\} + \sum_{j= i+2}^{n+1} \{a_j\}$ for some $c \in C_k$ and $a,a_{i+2},\ldots,a_{n+1} \in A$
$$i\cdot \{a\} + \{a + c\} + \{a+ c'\} + \sum_{j= i+3}^{n+1} \{a_j\}$$ 
for some $a,a_{i+3},\ldots,a_{n+1} \in A$ and $c,c' \in C_k$ such that $(i+2) \ a + c + c' + \sum_{j= i+3}^{n+1} a_j = 0$.
%\end{enumerate}
\end{lem}

\begin{proof}
%For $i =0$, Lemma~\ref{lem-rep} follows from the surjectivity of
%\begin{equation}
%\begin{aligned}
   %\ga_k :  \  C_k^2 \times A & \to A^2 \\
    %(c,c',a) & \mapsto (c+a,c'+a).
 %\end{aligned}
 %\end{equation}
%\end{cases}
Suppose  $z =  \sum_{j= 1}^{i+2} \{a + c_j\} +\sum_{j= i+3}^{n+1} \{a_j\}$ for some $c_1,\ldots,c_{i+2} \in C_k$. The cycle $z$ is rationally equivalent to $i\cdot \{a\} + \{a + c\} + \{a+ c'\} + \sum_{j= i+3}^{n+1} \{a_j\}$ where $c,c' $ are elements in $C_k$ such that $c + c' =\sum_{j= 1}^{i+2} c_j $ by Lemma~\ref{lem-critsum}.
 %, which proves $i)$. For $ii)$, the same argument shows that $z$ is rationally equivalent to $(i+1)\cdot \{a + c\} + \{a + c'\}  + \sum_{j= i+3}^{n+1} \{a_j\}$ where $c,c' \in C_k$ such that $(i+1) c + c' = \sum_{j= 1}^{i+2} c_j $.

\end{proof}

\section{The support of $S_i\CH_0(X)$}\label{sec-supp}

The subvarieties $V_{i,k}$ that we constructed in the previous section have the following property, whose proof will occupy the whole Section~\ref{sec-supp}. Recall that $\nu : K_n \to K_{(n)}$ is the Hilbert-Chow morphism. 

\begin{thm}~\label{thm-meet}
If $Z \subset K_n \subset A^{[n+1]}$ is a subvariety of dimension $i$ such that the zero-cycles in $A$ parameterized by $Z$ are rationally equivalent in $A$ to each other, then for some $k \in \bZ_{>0}$, there exist $x \in \nu^{-1}(V_{i,k})$ and $z \in Z$ such that $x$ and $z$ are  rationally equivalent in $K_n$.
\end{thm}

\begin{rem}
Naïvely, since 
$$\dim Z + \dim V_{i,k} = \dim K_n,$$ 
the subvarieties $Z$ and $\nu^{-1}(V_{i,k})$ are expected to have nonempty intersection, which would imply Theorem~\ref{thm-meet}. Part of the argument in the proof  establishes directly this nonemptiness in some situations (\emph{cf}. Subsection~\ref{ssec-nonsurj}). See also Remark~\ref{rm-nonpos} below.
\end{rem}

\begin{proof}[Proof of Theorem~\ref{thm-meet}]
The structure of the proof is inspired by Voisin's proof of~\cite[Theorem $2.1$]{Voisin0cycleK3}. Up to taking an irreducible component of $Z$, we suppose that $Z$ is \emph{irreducible}. The case $i = 0$ is trivial; below we will assume $i > 0$.
\ssec{ Reduction to the open multiplicity-stratum}
\begin{lem}~\label{lem-maxstrat}
It suffices to treat the case where a general element $z$ in $Z$ lies in the open multiplicity-stratum $A^{[n+1]}_{\red}$ parameterizing reduced subschemes of $A$.
\end{lem}

\begin{proof}

Assume the conclusion of Theorem~\ref{thm-meet} for all subvarieties in $K_n$ parameterizing zero-cycles in $A$ of constant class modulo rational equivalence and satisfying the condition in Lemma~\ref{lem-maxstrat}. Let $Z$ be a subvariety of $K_n$ as in the theorem. Suppose that a general element $z$ in $Z$ lies in the multiplicity-stratum $A^{[n+1]}_\mu$ for some partition 
$$\mu = 1^{\ga_1}\cdots (n+1)^{\ga_{n+1}}$$
 of $n+1$. Consider
 
 %We will prove Theorem~\ref{thm-meet} by decreasing induction on $(\ga_1,\ldots,\ga_{n+1})$ with the lexicographic order\footnote{Precisely, $(\ga_1,\ldots,\ga_{n+1}) < (\gb_1,\ldots,\gb_{n+1})$ if and only if there exists $1 \le i \le n+1$ such that $\ga_j = \gb_j$ for all $j < i$ and $\ga_i < \gb_i$.} starting with $(n+1,0,\ldots,0)$, for which the conclusion of Theorem~\ref{thm-meet} holds by assumption. 

%Suppose that $\ga_j > 0$ for some $j >1$. 

%If $j = n+1$, then $(\ga_1,\ldots,\ga_{n+1}) = (0,\ldots,0,1)$. In this case an element $z \in \nu(Z)$ represents a zero-cycle in $A$ of the form $(n+1)\{a\}$ where $a \in A[n+1]$. So  $z \in V_{n,n+1}$.

 %If $j \ne n+1$, then there exists $l \ne j$ such that  $\ga_l > 0$. 
 
 \begin{equation*}
Z_1 \colonec\left\{ 
   \sum_{j=1}^{n+1} \sum_{1 \le p \le \ga_j} \sum_{q=1}^j\{c_{j,p,q}+a_{j,p}\}  \in K_{(n)} \  \Bigg{|} 
  % \begin{split}
        \  a_{j,p} \in A, \ c_{j,p},c_{j,p,q}\in C, \ \sum_{q=1}^j c_{j,p,q} = j\cdot c_{j,p}, 
        \ \sum_{j=1}^{n+1} \sum_{1 \le p \le \ga_j} j \{c_{j,p} + a_{j,p}\} \in \nu(Z)  
 %\end{split}
 \right\},
  \end{equation*}
  where we recall that the sum of elements within (resp. without) curly brackets is the sum of zero-cycles (resp. defined by the group law in $A$). By Lemma~\ref{lem-critsum}, $Z_1$ parameterizes the same class of zero-cycles in $A$ as $Z$ parameterizes. On one hand, it is easy to see that 
  $$\dim Z_1 = \dim \nu(Z) + \sum_{j=1}^{n+1}\ga_j ( j - 1).$$ 
  On the other hand, %since a general element in $Z_1$ lies in the multiplicity-stratum $A^{[n+1]}_{\mu'}$ with
  %$$\mu' = 1^{\ga_1 + j}2^{\ga_2}\cdots j^0 (j+1)^{\ga_{j+1}} \cdots (n+1)^{\ga_{n+1}} > \mu,$$
 we see by~\cite{Briancon} that if $z$  is a general element in $\nu(Z)$, then 
 $$\dim \nu^{-1}(z) = \sum_{j=1}^{n+1}\ga_j ( j - 1).$$
 So if $\wt{Z_1}$ denotes  the strict transform  of $Z_1$ under $\nu : K_n \to K_{(n)}$, then $\dim \wt{Z_1} = \dim Z = i$ and a general element in $\wt{Z_1}$ lies in  the open multiplicity-stratum $A^{[n+1]}_{\red}$. By assumption, there exist $x \in \nu^{-1}(V_{i,k})$ and $z \in \wt{Z_1}$ such that $x$ and $z$ are  rationally equivalent in $K_n$. Finally by Lemma~\ref{lem-critsum} and the definition of $Z_1$, there exists $z' \in Z$ which is rationally equivalent to $z$ in $K_n$, hence to $x$.

%Let $Z^\circ \colonec \nu(Z) \cap A^{(n+1)}_\mu$ where $\nu :A^{[n+1]} \to A^{(n+1)}$ is the Hilbert-Chow morphism; an element of $Z^\circ$ will be denoted by $z = i\{z_1\} + \cdots + i\{z_{\ga_i}\}  + z'$ where $z'$ is a zero-cycle supported on $A \bss 
%\{z_1,\ldots,z_{\ga_i}\}$ with coefficients $\ne i$.  
%Define
%\begin{equation}\label{eq-orbitmap}
%\begin{split}
%\psi : Z^\circ \times C  & \to A^{(n+1)} \\
 %\( i\{z_1\} + \cdots + i\{z_{\ga_i}\}  + z', c \) & \mapsto    \sum_{ 1 \le k \le \ga_i} \(\{z_k + c\} + \{z_k - c\} + \(i-2\)\{z_k \}\) + z'.
%\end{split}
%\end{equation}
%so that $Z' \colonec \ol{\nu^{-1}(\Ima \psi)}$ parameterizes zero-cycles in $A$ of the same class as the one parameterized by $Z$. Since $z_1,\ldots,z_{\ga_i}$ are all different and $z_j \notin \supp(z')$ for all $j$, a general element in $Z'$ lies in the multiplicity-stratum $A^{[n+1]}_{\mu'}$ where 
%$$\mu' = 1^{\ga_1 +2\ga_i}2^{\ga_2}\cdots(i-2)^{\ga_{i-2} + \ga_i}(i-1)^{\ga_{i-1}}i^0(i+1)^{\ga_{i+1}}\cdots(n+1)^{\ga_{n+1}}  > \mu.$$ 
% Hence by induction hypothesis, since $\dim Z' \ge i$, there exists a point $z' \in Z'$ such that $z'$ is rationally equivalent in $K_n$ to some point in $V_{i,k}$ for some $k$. Finally since $E \simeq \bP^1$, $\nu(z')$ is rationally equivalent \emph{in ${\nu(K_n)}$} to $\nu(z)$ for some point $z$ in $Z$. As the fibers of $\nu$ are rationally connected, we conclude that $z \sim_{\rat} z'$ in $K_n$.
\end{proof}

\ssec{Setups}
\hfill

By virtue of Lemma~\ref{lem-maxstrat}, we may and we will assume that a general element $z$ in $Z$ lies in the open multiplicity-stratum $A^{[n+1]}_{\red}$. In particular, $\dim Z = \dim \nu(Z)$. Since the Hilbert-Chow morphism $\nu : K_n \to K_{(n)}$ induces an isomorphism $\nu_* : \CH_0(K_n) \eto \CH_0(K_{(n)})$, Theorem~\ref{thm-meet}, it suffices to prove the following analogue version of Theorem~\ref{thm-meet} in $K_{(n)}$.

\begin{thmbis}{thm-meet}
If $Z \subset K_{(n)} \subset A^{(n+1)}$ is a subvariety of dimension $i$ such that the zero-cycles in $A$ parameterized by $Z$ are rationally equivalent in $A$ to each other, then for some $k \in \bZ_{>0}$, there exist $x \in V_{i,k}$ and $z \in Z$ such that $x$ and $z$ are  rationally equivalent in $K_{(n)}$.
\end{thmbis}

We will prove Theorem~\ref{thm-meet}'   by induction on $n \ge 1$. For $n=1$, the only case to prove is that of $i=1$. Note that $K_{(2)} \to A / \imath$ associating $(a,-a) \in K_{(2)}$ to the class of $a$ under  the involution action $\imath$ is an isomorphism. \emph{Via} this isomorphism,  $V_{1,1}$ is the image of the theta divisor $C \subset A$ under the quotient map $A \to  A / \imath$, so $V_{1,1}$ is ample. As $\dim Z \ge 1$, $Z \cap V_{1,1}$ is not empty, which proves Theorem~\ref{thm-meet}' in this case. 

From now on we assume $n > 1$. Let $Z'$ be one of the \emph{irreducible} components of $q_{n+1}^{-1}\(Z\)$ where we recall that $q_{n+1} : A^{n+1} \to A^{(n+1)}$ denotes the quotient map. Let (H) denote the assumption
\begin{equation}\label{hyp}
\textit{There exists an integer $j$ such that the image of $Z'$ under the $j$-th projection $A^{n+1} \to A$ is $A$.} 
\tag{H}
\end{equation}

\ssec{Proof of Theorem~\ref{thm-meet}' under  induction hypothesis and assumption (H)} \hfill

In this paragraph, we assume that $Z'$ verifies (H). If we define
\begin{equation}
\begin{split}
p_j :  A^{n+1} & \to A^{2} \\
  \(a_1,\ldots,a_{n + 1} \) & \mapsto   \(a_j,  \sum_{l \ne j}a_l \),
\end{split}
\end{equation}
assumption (H) implies that ${p_j}_{|Z'} : Z'  \to A^2_0$ is surjective; without loss of generality we can assume $j=1$. For simplicity, ${p_{1}}$ will be denoted  by $p$  from now on until the end of the proof.  Define the map
\begin{equation}
\begin{split}
\hat{p} :  A^{n+1} & \to A^{n} \\
 \(a,a_1,\ldots,a_{n} \) & \mapsto   \(a +n\cdot a_1,\ldots,a + n \cdot a_{n} \).
\end{split}
\end{equation}

\begin{lem}\label{lem-genfin}
In the situation above, the map $\hat{p}_{|Z'}$ is generically finite. 
\end{lem}

\begin{proof}

First of all, let 
$$\Gamma \colonec \{(a_1,\ldots,a_{n+1},a) \mid a_1,\ldots,a_{n+1} \in A, \ a = a_i \text{ for some } i \ \} \subset A^{n+1} \times A$$ 
denote the incidence correspondence. Since $Z' \subset A^{n+1}$ parameterizes zero-cycles of constant class in $\CH_0(A)$, we see by~\cite[Proposition 10.24]{VoisinII} that for all $\ga,\gb \in H^0(A,\gO_A^1)$, 
\begin{equation}\label{eqn-vanform}
\Gamma^*\ga = \sum_{i = 1}^{n+1}\(\pr_i^*\ga\)_{|Z} = 0 \ \ \ \ \text{ and } \ \ \ \ \Gamma^*(\ga \wedge \gb) = \sum_{i = 1}^{n+1}\(\pr_i^*(\ga \wedge \gb)\)_{|Z} = 0,
\end{equation}
where $\pr_i : A^{n+1} \to A$ is the $i$-th projection.

Next by definition of $\hat{p}$, if $\gs \colonec \sum_{i=1}^{n} \pr_i^*(\ga \wedge \gb) \in H^0(A^n, \gO_{A^n}^2)$, then elementary computations show that
$$\hat{p}^*\gs = \sum_{i = 2}^{n+1} \(\pr_1^*(\ga \wedge \gb) + n^2 \cdot \pr_i^*(\ga \wedge \gb)  + n\cdot \pr_1^*\ga \wedge \pr_i^*\gb - n\cdot \pr_1^*\gb \wedge \pr_i^*\ga\).$$
The above formula together equations~\eqref{eqn-vanform} yield
\begin{equation}\label{eqn-simp}
\hat{p}^*\gs_{|Z'}  = (n - n^2)\cdot \pr_1^*(\ga \wedge \gb)_{|Z'}.
\end{equation}
Here we recall that $n > 1$, so $n- n^2 \ne 0$.

Now choose $\ga,\gb \in H^0(A,\gO_A^1)$ so that $\ga \wedge \gb \in H^0(A,\gO_A^2)$ is non-degenerated. Recall that ${\pr_1}_{|Z'}$ is surjective by assumption (H). It follows that if $z \in Z'$ is a smooth general point, the differential ${(\pr_1)}_{|Z'*}$ is surjective at  $z$  and thus the kernel of the two-form ${(\pr_1)}_{|Z'}^*(\ga \wedge \gb)$ is equal to  $\ker({(\pr_1)}_{|Z'})_*$. On the other hand,  formula~\eqref{eqn-simp} shows that if $u \in T_{Z',z}$ is annihilated by $\hat{p}_*$, then $u \in \ker (\pr_1^*(\ga \wedge \gb)_{|Z'})$, since $u \in \ker(\hat{p}^*\gs_{|Z'})$.  Therefore $u \in  \ker (\hat{p}_*) \cap \ker({\pr_1}_*) = \{0\}$, hence $\hat{p}_{|Z'}$ is generically finite.

\end{proof}

Before we continue, let us prove a general formula.

\begin{lem}\label{lem-genfin1}
For $a,a_1,\ldots,a_n \in A$, the following equality holds in $\CH_0(A)$:
\begin{equation}\label{eqn-Blochn}
\sum_{j= 1}^n\{a_j + a\} = \left\{ a + \sum_{j= 1}^na_j \right\} + \sum_{j= 1}^n\{a_j \} - \left\{\sum_{j= 1}^na_j \right\} + (n-1)\(\{a\} - \{0\}\).
\end{equation}
\end{lem}
\begin{proof}

Let us recall for convenience the following formula due to Bloch~\cite[Theorem $(0.1)$, case $n = 2$]{Blochab}. If $a,b,c \in A$ where $A$ is an abelian surface with origine $o$, then the following holds in $\CH_0(A)$: 

\begin{equation}\label{eqn-Bloch}
\{ o\} - \{ a\} - \{ b\} - \{ c\} + \{a+b \} + \{b+c \} + \{c+a \} - \{a +b+c \} = 0.
\end{equation}

We will prove equality~\ref{eqn-Blochn} by induction starting from $n =1$ and $2$. When $n=1$, there is noting to prove. A direct application of Bloch's formula~\eqref{eqn-Bloch} yields the case $n = 2$, from which we deduce the following equality for $n>1$ in $\CH_0(A)$:
$$\{a_n + a\} + \left\{ a + \sum_{j= 1}^{n-1}a_j \right\} = \left\{ a + \sum_{j= 1}^{n}a_j \right\} + \left\{\sum_{j= 1}^{n-1}a_j \right\}+  \{a_n\} - \left\{\sum_{j= 1}^{n}a_j \right\} + \(\{a\} - \{0\}\).$$
Lemma~\ref{lem-genfin1} thus follows easily from induction hypothesis.
\end{proof}

\begin{lem}\label{lem-RErec}
Let $\Delta^{\#}_{C_k} \colonec \{(c,- c) \in A^2 \mid c \in C_k\}$. The subvariety $Z_{C_k} \colonec \hat{p}\(p^{-1}(\Delta^{\#}_{C_k}) \cap Z'\)$ parameterizes effective zero-cycles of degree $n$ in $A$ of constant class modulo rational equivalence. 
\end{lem}

\begin{proof}
Every equality appearing in this proof holds in $\CH_0(A)$. Let $(c,a_1,\ldots,a_n), (c',a'_1,\ldots,a'_n) \in p^{-1}(\Delta^{\#}_{C_k}) \cap Z'$. %Note that if $p(c,a_1,\ldots,a_n) = p(c'',b_1,\ldots,b_n)$, then the difference between $c$ and $c''$ is $n$-torsion. So $\hat{p}(c,a_1,\ldots,a_n)$ and $\hat{p}(c'',b_1,\ldots,b_n)$ represent the same class of zero-cycles in $\CH_0(A)$. Thus we may suppose that $$. 
Note that by formula~(\ref{eqn-Blochn})
\begin{equation}\label{eqn-1}
\sum_{j= 1}^n\{n\cdot a_j + c\} = \left\{ c -n \cdot c \right\} +\( \sum_{j= 1}^n\{n\cdot a_j \} + \{n \cdot c\}\) - \(\{n \cdot c\} + \left\{ - n \cdot c\right\} \)+ (n-1)\(\{c\} - \{0\}\).\end{equation}
Since $c \in C_k$, one has
\begin{equation}\label{eqn-2}
\left\{ c -n \cdot c \right\}+ (n-1)\(\{c\} - \{0\}\) = 0
\end{equation}
For the same reason,
 \begin{equation}\label{eqn-3}
 \{n \cdot c\} + \left\{ - n \cdot c\right\} = \{n \cdot c'\} + \left\{ - n \cdot c'\right\}.
\end{equation}
Since $n\cdot Z$ parameterizes zero-cycles in $A$ of constant rational equivalence class, we see that 
\begin{equation}\label{eqn-4}
\{n\cdot c\}+ \sum_j\{n\cdot a_j\} = \{n \cdot c'\}+ \sum_j\{n \cdot a'_j\}.
\end{equation}
Combining identities~(\ref{eqn-1}),~(\ref{eqn-2}),~(\ref{eqn-3}), and~(\ref{eqn-4}), we deduce that 
$$\sum_{j= 1}^n\{n\cdot a_j + c\} = \sum_{j= 1}^n\{n\cdot a'_j + c'\}.$$
\end{proof}

\begin{proof}[Proof of Theorem~\ref{thm-meet}' under the assumption (H)]
Recall that assumption (H) says that ${p}_{|Z'} : Z'  \to A^2_0$ is surjective. On one hand, since $\hat{p}_{|Z'}$ is generically finite by Lemma~\ref{lem-genfin} and since the union $\cup_{k\in \bZ} \Delta^{\#}_{C_k}$ is Zariski dense in $A^2_0$, there exists $l \in \bZ_{>0}$ such that $Z_{C_l}$ is of dimension $i-1$. On the other hand $Z_{C_l} \subset A^{n}_0 $, and by Lemma~\ref{lem-RErec},  $Z_{C_l}$  parameterizes effective zero-cycles of degree $n$ in $A$ of constant class modulo rational equivalence, we can  apply induction hypothesis on $Z_{C_l}$.

If $i <n$, induction hypothesis shows that there exist $a',a'_{i+2},\ldots,a'_n \in A$ and some $k \in \bZ_{>0}$ such that each element in $q_n\(Z_{C_l}\)$ is rationally equivalent \emph{in $K_{(n-1)}$} to 
$$\sum_{j=1}^{i+1}\{ a' + c_j\}+ \sum_{j=i+2}^{n}\{a' +  a_j'\},$$
for all $c_1,\ldots,c_{i+1} \in C_k$ such that $ n \cdot a'  + \sum_{j=1}^{i+1} c_j + \sum_{j=i+2}^{n}a'_j = 0$. As $\phi : C_k \times C_l \to A$ defined by $\phi(c,c') \mapsto c- (n+1)\cdot c'$ is surjective, there exist $c_0 \in C_k$ and $c \in C_l$ such that $n\cdot c = c_0 - c + a'$. Therefore for any $(c,a_1,\ldots,a_n) \in p^{-1}(\Delta_{C_l}) \cap Z'$, whose existence is due to the surjectivity of $p_{|Z'} : Z' \to A^2_0$, the following equality holds in $ \CH_0(K_{(n)})$:
$$\{n \cdot c\} + \sum_{j=1}^n\{n \cdot a_j\} =\sum_{j=0}^{i+1}\{ (a' - c) + c_j\}+ \sum_{j=i+2}^{n}\{ (a' - c) + a_j'\}.$$
Thus if $z \in Z$ satisfies $\nu\(n\cdot z\) = \{n \cdot c\} + \sum_{j=1}^n\{n \cdot a_j\}$,  there exists $z' \in V_{i,nk} $ such that $\nu(z) \sim_\rat z'$ in $K_{(n)}$.

For the remaining case $ i = n$, applying induction hypothesis as before, there exists $c' \in C_k$ for some $k \in \bZ_{>0}$ such that every point in $q_n\(Z_{C_l}\)$ is rationally equivalent in $K_{(n-1)}$ to  $\sum_{j=1}^n \{c' +c_j\}$
for all $c_1,\ldots,c_{n} \in C_k$ such that $n \cdot c' + \sum_{j=1}^{n} c_j  = 0$. The same argument above replacing $a'$ with $c'$ allows to conclude.
\end{proof}

\ssec{General case}\label{ssec-nonsurj}\hfill

Now assume that $Z'$ does \emph{not} verify (H). Then there exist curves $D_1,\ldots D_{n+1} \subset A$ such that 
\begin{equation}\label{pte-inclusion}
Z'  \subset \prod_{j = 1}^{n+1} D_j.
\end{equation}
As $Z'$ is irreducible, up to removing some irreducible components of $D_j$, we can suppose that  $D_j$ is irreducible for all $j$.

\begin{rem}\label{rm-nonpos}

One would expect to prove Theorem~\ref{thm-meet}' in this case by showing directly that 
\begin{equation}\label{eqn-nonvidereve}
\tau(Z',A) \cap\( C^{i+2}_k \times A^{n-1-i}\)  \ne \emptyset
\end{equation}
where $\tau$ is defined in~\eqref{def-tau1}, merely under the assumption~(\ref{pte-inclusion}) just by positivity arguments as in Voisin's proof of~\cite[Theorem $2.1$]{Voisin0cycleK3}.  However the following example shows that the above expectation fails in some cases. Take for example $i = 1, n=2$ and set 
$$Z' \colonec \{(c,c + a,c + a') \mid c\in C_k\} \subset A^3$$
for some fixed $a , a' \in A$. If $a$ and $a'$ are generically chosen, then there is no $c \in C_k$ such that $c+a,c+a' \in C_k$. In other words, $\tau(Z',A) \cap C^{3}_k  = \emptyset$.

We will \emph{not} prove the non-emptiness~(\ref{eqn-nonvidereve}) for any $Z'$ which does not satisfy hypothesis (H). Instead, for those $Z'$ such that~(\ref{eqn-nonvidereve}) might fail, we will reduce the proof of Theorem~\ref{thm-meet}'  to the situation where hypothesis (H) is verified.

\end{rem}
%Since $Z'$ is irreducible, $C'$ is also irreducible.

\sssec{Case $i = n$}\hfill

Under hypothesis~(\ref{pte-inclusion}), we will first prove Theorem~\ref{thm-meet}' for the Lagrangian case, that is for $i = \dim Z' = n$. Since  $Z'  \subset  \prod_{j = 1}^{n+1} D_j \cap A^{n+1}_0$, we see that $\dim  A^{n+1}_0 \cap \prod_{j = 1}^{n+1} D_j = n$.

\begin{lem}
If $\dim  A^{n+1}_0 \cap \prod_{j = 1}^{n+1} D_j = n$ and the image of $\prod_{j = 1}^{n+1} D_j$ under the sum map $\mu : A^{n+1} \to A$ is  $A$, then $n=1$.
\end{lem}

\begin{proof}
Suppose that $n > 1$. Since $\dim A^{n+1}_0 \cap \prod_{j = 1}^{n+1} D_j = n$, there exists a  projection from $\prod_{j = 1}^{n+1} D_j$ onto a product of $n-1$ factors whose restriction (denoted $r$) to $\prod_{j = 1}^{n+1} D_j \cap A^{n+1}_0$ is surjective; without loss of generality we can suppose $r$ to be the projection $ \prod_{j = 1}^{n+1} D_j \cap A^{n+1}_0\to \prod_{j = 1}^{n-1} D_j$. Since a general fiber of $r$ is one-dimensional, for a general $(n-1)$-uple $(c_1,\ldots,c_{n-1}) \in   \prod_{j = 1}^{n-1} D_j$ and for any $c \in D_n$, there exists $c' \in D_{n+1}$ such that
$$c + c' + \sum_{j=1}^{n-1} c_j = 0,$$
which is impossible unless $n =1$.
\end{proof}

\begin{lem}\label{lem-i=n}
If the image of $\prod_{j = 1}^{n+1} D_j$ under the sum map $\mu : A^{n+1} \to A$ is of dimension $< 2$, then there exists some $k \in \bZ_{>0}$ such that,
$$ Z' \cap \tau(C_k^{n+1},C_k) \ne \emptyset.$$
In particular, Theorem~\ref{thm-meet}' holds when $i = n$ and $Z'$ does not satisfy hypothesis (H).
\end{lem}

\begin{proof}
By the assumption of Lemma~\ref{lem-i=n}, for all $1 \le j \le n + 1$ and $(c_1,\ldots,c_{n+1}) \in \prod_{j = 1}^{n+1} D_j$,
$$\sum_{l = 1}^{n+1} D_l = \(\sum_{l \ne j } c_l\) + D_j.$$
It follows that there exists an elliptic curve $E_0 \subset A$ such that the $D_j$'s are translations of $E_0$. Thus $ A^{n+1}_0 \cap \prod_{j = 1}^{n+1} D_j $ is irreducible, so $Z' =  A^{n+1}_0 \cap \prod_{j = 1}^{n+1} D_j$.

%for any $c \in C'$, $\tau_{-c}(Z') \subset (C' -c)^{n+1}$ is of dimension $n$. So there exists a standard projection $(C' - c)^{n+1} \to ( C' - c)^{n}$ whose restriction (denoted $r_1$) to $Z'$ is surjective, and in particular, $r_1^{-1}(0,\ldots,0) \ne \emptyset$. Since $Z' \subset A^{n+1}_0$, we deduce that $(c,\ldots,c,-n\cdot c) \in Z'$. 

If $A = E \times E'$, we can suppose without loss of generality that $\tau_a(E_0) \cap \(E \times \{0\}\) \ne \emptyset$ for all $a \in A$; choose $c_j \in D_j \cap \(E \times \{0\}\)$ for all $0 \le j \le n$. Since $0 \in \sum_{j = 1}^{n+1} D_j = \(\sum_{ j = 1}^n c_l\) + D_{n+1}$, there exists $c_{n+1} \in D_{n+1}$ such that $\sum_{j=1}^{n+1} c_j = 0$. Since $c_1,\ldots,c_n \in E \times \{0\}$, we see that $c_{n+1} \in E \times \{0\}$, hence $Z \cap C^{n+1} \ne \emptyset$.

In the case where  $A$ is the Jacobian of a smooth curve,  $C_k$ is not contained in any of the translates of $E_0$. Accordingly,  since
$$F \colonec \left\{-c - \sum_{j=1}^n(c+c_j) \ \big{|} \ c \in C_k, \ c+c_j  \in (c+C_k) \cap D_j \right\} \subset  C_k + E'_0$$
where $E'_0$ is some translation of $E_0$, $F$ is of  dimension $>0$. So  $F \cap C_k$ is non-empty since $C_k$ is ample. Therefore there exist $c,c_1,\ldots, c_{n} \in C_k$ such that $c+c_j  \in (c+C_j) \cap D_j$ and $c_{n+1} \colonec -c - \sum_{j=1}^n(c+c_j)  \in C_k$, so
$$(c + c_1, \ldots, c + c_{n+1}) \in A^{n+1}_0 \cap \prod_{j = 1}^{n+1} D_j = Z'.$$
Hence $ Z' \cap \tau(C_k^{n+1},C_k) \ne \emptyset$.

Finally, choose $z \in Z' \cap \tau(C_k^{n+1},C_k)$. Since $q_{n+1}(Z') = Z \subset K_{(n)}$ and  $V_{n,k} = q_{n+1}(\tau(C_k^{n+1},C_k)) \cap K_{(n)}$, we see that $q_{n+1}(z) \in Z \cap V_{n,k}$. Hence Theorem~\ref{thm-meet}' is proven in this case.

%By assumption, $C' + C' \subset A$ is an irreducible curve, so equals $c + C'$ for any $c \in C'$. It follows that for any $c \in C'$, $C' - c$ is an elliptic curve. Since $C_k$ is ample, we can choose $c \in C_k \cap C'$.

%Since $C'$ is irreducible, $\mu(C'^{n+1})$ is irreducible and containing $n\cdot c + C'$ for any $c \in C'$, so $\mu(C'^{n+1}) = n\cdot c + C'$. As $Z'_1  \subset (C')^{n+1} \cap A^{n+1}_0$, we see that $0 \in \mu(C'^{n+1})$.
\end{proof}
\sssec{Case $i < n$}\hfill

The following lemma allows to conclude the proof of Theorem~\ref{thm-meet}' for the remaining case. 

\begin{lem}\label{lem-i<n}
If $i<n$ and $Z'$ satisfies hypothesis~(\ref{pte-inclusion}), then either the non-emptiness~(\ref{eqn-nonvidereve}) holds, or there exists $Z_1'' \subset A^{n+1}_0$ of dimension $i$ such that  $Z_1''$ verifies hypothesis (H) and all points of $Z_1''$ represent zero-cycles in $A$ of the same rational equivalence class in $A$  as zero-cycles  parameterized by  $Z'$. 
\end{lem}

\begin{proof}
Since $C_k$ is ample and $Z' \subset \prod_{j=1}^{n+1} D_j$, we see that for all $a \in A$,
$$[Z'] \cdot \(\sum_{j_1<\cdots < j_i}\prod_{l=1}^i\pi_{j_l}^*[\tau_a(C_k)]\) = [Z'] \cdot \(\sum_{j=1}^{n+1}\pi_j^*[\tau_a(C_k)]\)^i \ne 0, $$
where $\pi_j : A^{n+1} \to A$ denotes the $j$-th standard projection. Thus for all $a \in A$, up to permutation of factors, 
\begin{equation}\label{eqn-nonemptyZ'cap}
Z' \cap \tau_{a}\(C_k^i \times A^{n-i+1}\) \ne \emptyset
\end{equation}

Next consider  
\begin{equation*}
Z'' \colonec \left\{ (a_1,\ldots,a_{n-i+1}) \in A^{n-i+1} \ \big{|} \  \tau_a(c_1,\ldots,c_i,a_1,\ldots,a_{n-i+1}) \in Z' \cap \tau_{a}\(C_k^i \times A^{n-i+1}\) 
 \text{ for some } a, c_1,\ldots,c_i \in A \right\}.
\end{equation*}
Since $Z' \subset \prod_{j=1}^{n+1} D_j$, by~(\ref{eqn-nonemptyZ'cap})  the first projection of $\pi : Z'' \to A$ is not constant; in particular $\dim Z'' > 0$. Accordingly  $\pi^{-1}(C_k) \subset Z''$ has codimension $\le 1$. Thus if $\dim Z'' \ge 2$,  one of the standard projection $\pi^{-1}(C_k) \to A$ has positive dimension hence must intersect $C_k$. It follows that up to permutation of factors, 
$$Z' \cap \tau\(C_k^{i+2} \times A^{n-i-1}, A\) \ne \emptyset.$$
So if we choose $z \in Z' \cap  \tau\(C_k^{i+2} \times A^{n-i-1}, A\)$, since $q_{n+1}(Z') = Z \subset K_{(n)}$ and  $V_{i,k} = q_{n+1}(\tau(C_k^{i+2} \times A^{n-i-1}, A)) \cap K_{(n)}$, we see that $q_{n+1}(z) \in Z \cap V_{i,k}$. Thus Theorem~\ref{thm-meet}' is proven in this case.

If $\dim Z'' = 1$, then by~(\ref{eqn-nonemptyZ'cap}), for all $\ul{a} \colonec (a_1,\ldots,a_{n-i+1}) \in Z''$, there exists a curve $D_{\ul{a}} \subset A$ such that for all $a \in D_{\ul{a}}$, there exist  $c_1,\ldots,c_i \in C_k$ such that
$$\tau_a(c_1,\ldots,c_i,a_1,\ldots,a_{n-i+1}) \in Z' \cap \tau_{a}\(C_k^i \times A^{n-i+1}\).$$
It follows that for all $1 \le j\le  n-i+1$,
\begin{equation}\label{translinv}
a_j + D_{\ul{a}} = D_{i+j}.
\end{equation}
On the other hand, as $\tau_a(c_1,\ldots,c_i,a_1,\ldots,a_{n-i+1})  \in Z' \subset A^{n+1}_0$,
\begin{equation}\label{agacant}
(n+1)\cdot a+ \sum_{j=1}^{i} c_j + \sum_{l=1}^{n-i+1} a_l = 0.
\end{equation}
for all $a \in D_{\ul{a}}$. So if $i = 1$, we see that  a translation of  $-(n+1) \cdot D_{\ul{a}}$ is contained in $C_k$.
Since $D_{\ul{a}}$ does not depend on $k$ by~(\ref{translinv}), we deduce that if $A$ is a jacobian of a smooth curve, there exists $k \in \bZ_{>0}$ such that $\dim Z'' \ne 1$. Hence $Z' \cap \tau\(C_k^{3} \times A^{n-2}, A\) \ne \emptyset$ for such a $k$, so Theorem~\ref{thm-meet}' is proven in this situation by the same argument above. Still in the case where $i =1$, if $A = E \times E'$, without loss of generality we can suppose that $D_{\ul{a}}$ is a translation of $E \times\{0\}$. So for each $1 \le j \le n+1$, $D_j$ is also a translation of $E \times\{0\}$. Since $\dim Z'  >0$ and $Z' \subset A^{n+1}_0$, there exists $(x_1,\ldots,x_{n+1}) \in Z' \subset A^{n+1}$ such that the projection of $x_j - x_l \in A$ to $E$ is $k$-torsion for some $j$ and $l$; without loss of generality we can assume $j =2$ and $l = 3$. If $y$ denotes the projection of $x_2 \in A$ onto $E$, we see that
$$\tau_{- x_1 -(y,0)}(x_1,\ldots,x_{n+1}) \in C_k^3 \times A^{n-2}.$$
Hence we also have  $ Z' \cap \tau(C_k^3 \times A^{n-2},A) \ne \emptyset$.

There remains the case where  $i > 1$. Since $C_k$ is ample, without loss of generality there exists $\ul{a} \colonec (a_1,\ldots,a_{n-i+1}) \in Z''$ such that $a_1 \in C_k$. 
Define
\begin{equation*}
Z''_1 \colonec \left\{ \tau_a(c'_1,\ldots,c'_{i+1},a_2,\ldots,a_{n-i+1}) \in A^{n+1} \  \Bigg{|} 
\begin{split}
&\  \sum_{j=1}^{i+1} c'_j =a_1 + \sum_{j=1}^i c_j, \text{ for some } c_1,\ldots,c_i \in C_k \text{ such that }  \\
&  \  \tau_a(c_1,\ldots,c_i,a_1,\ldots,a_{n-i+1}) \in Z' \cap \tau_{a}\(C_k^i \times A^{n-i+1}\)  \text{ for some } a \in D_{\ul{a}} 
  \end{split} 
  \right\}.
\end{equation*}
On one hand, by Lemma~\ref{lem-critsum}, the zero-cycles in $A$ parameterized by $Z''_1 \subset A^{n+1}_0$ have the same class in $\CH_0(A)$ as the one parameterized by $Z'$. On the other hand, it is easy to see that $\dim Z''_1 = i$. If $i > 1$, then the image of $Z''_1$ under some standard projection $A^{n+1} \to A$ is $A$. We conclude that $Z''_1$ satisfies hypothesis (H).
%Thus up to replacing $Z'$ by $Z''_1$, we have reduced the situation to the case where the restriction to $Z'$ of some standard projection $A^{n+1} \to A$ is surjective, which will be treated in the next subsection.
\end{proof}

\begin{rem}
The properties~(\ref{translinv}) and~(\ref{agacant}) imply that for all $a \in D_{\ul{a}}$ and $1 \le j \le i$, $c_j \in C_k \cap (D_j - a)$  and $\sum_{j=1}^{i} c_j$ belongs to a translation of  $(n+1) \cdot D_{\ul{a}}$. Given these restrictive conditions, it would be possible to conclude directly as in the case $i=1$ that there exists $k \in \bZ_{>0}$ such that $\dim Z'' \ne 1$.
\end{rem}

\begin{rem}
The reason why we distinguish the cases $i= n$ and $i<n$ in the proof is essentially because $V_{i,k} =  q_{n+1}(\tau(C_k^{i+2} \times A^{n-i-1}, A)) \cap K_{(n)}$ for $i<n$ and $V_{n,k}= q_{n+1}(\tau(C_k^{i+1}, C_k)) \cap K_{(n)}$ have different form: $V_{i,k}$ is first of all the union of all translates by elements in $A$ of $q_{n+1}(C_k^{i+2} \times A^{n-i-1})$ then intersected with $K_{(n)}$, whereas $V_{n,k}$ is only the union of all translates by elements \emph{in $C_k$} of $q_{n+1}(C_k^{n+1})$, then intersected with $K_{(n)}$.
\end{rem}

Theorem~\ref{thm-meet}' now results from the combination of the proof of  Theorem~\ref{thm-meet}' under hypothesis (H) and Lemmata~\ref{lem-i=n} and~\ref{lem-i<n}.

\end{proof}

%\begin{lem}
%Every $x \in V_{i,k}$ is rationally equivalent (in $K_n^\circ(A)$) to some point $x' \in V_{i,1}$.
%\end{lem}

The following result is an important corollary of Theorem~\ref{thm-meet}'.

\begin{cor}~\label{cor-main}
If $x$ is a point of $K_n$ such that $\dim O_x \ge i$, then $\nu(x)$ is rationally equivalent to some zero-cycle supported on $V_{i,k}$ for some $k$. In other words,
$$\nu_* S_i\CH_0(K_n) = \bigcup_{k \in \bZ_{>0}}\Ima \(\CH_0(V_{i,k}) \to \CH_0(K_{(n)})\).$$
\end{cor}

\begin{proof}
The incidence correspondence $\Gamma \subset K_n \times A$ sends the class of a point $z \in K_n$ to the class of the zero-cycle in $A$ by which $z$ represents. So a constant cycle subvariety in $K_n$ parameterizes   zero-cycles in $A$ of constant class modulo rational equivalence. We apply Theorem~\ref{thm-meet}' to conclude.
\end{proof}

\section{The induced Beauville decomposition on generalized Kummer varieties}\label{sec-Beauvillefiltr}

We define in this section another filtration on $\CH_0(K_n)$ coming from the  Beauville decomposition of an abelian variety. 

\ssec{Description of the Beauville decomposition}\hfill

Recall in~\cite{BeauvilleDecAb} that for any abelian variety $B$, the Chow ring (with rational coefficients) of $B$ has a canonical ring grading called the \emph{Beauville decomposition}
\begin{equation}\label{Bdecomp}
\CH^p(B) = \bigoplus_{s= p-g}^p \CH^p(B)_s
\end{equation}
for $0 \le p \le g \colonec \dim B$, where
$$\CH^p(B)_s \colonec \{z \in \CH^p(B) \mid [m]^*z = m^{2p - s}z \text{ for all } m \in \bZ \}.$$
Based on~\cite{BeauvilleDecAb,DeningerMurre,Kunnemann,Shermenev},  the Künneth decomposition of the cohomological class of the diagonal $[\Delta] = \sum_{j=1}^{2g}[\Delta_j] \in H^{2g}(B \times B,\bQ)$ (where $[\Delta_j]$ is the component inducing the identity map on $H^j(B,\bQ)$) lifts to a decomposition $\Delta = \sum_{j=1}^{2g}\Delta_j \in \CH^{g}(B \times B)$ such that ${\Delta_j}_*$ acts as the projector $\CH^p(B) \to \CH^p(B)_{2p-j}$ for all $p$~\cite[\S 2.5]{MurreFiltrBBM}. Such a decomposition of $\Delta \in \CH^{g}(B \times B)$ is called a Chow-Künneth decomposition of $B$. As the Beauville decomposition is multiplicative, by~\cite[Theorem 5.2]{JannsenFBBM} if the Bloch-Beilinson filtration $F_{BB}^\bullet$ on $\CH^\bullet(B)$ exists, then the Beauville decomposition would give a splitting of $F_{BB}^\bullet\CH^\bullet(B)$. 

Let $A^{n+1}_0$ denote the kernel of the sum map $\mu : A^{n+1} \to A$. The symmetric group $\fS_{n+1}$ acts on $A^{n+1}_0$ and the resulting quotient variety is $K_{(n)}$. Let $q : A^{n+1}_0 \to K_{(n)}$ denote the quotient map. Since $[m] : A^{n+1}_0 \to A^{n+1}_0$ commutes with the action of $\fS_{n+1}$ permuting the factors for each $m \in \bZ$, 
\begin{equation}\label{eqn-decomp}
q^* : \CH_0(K_{(n)}) \eto  \(\bigoplus_{s= 0}^{2n} \CH_0\(A^{n+1}_0\)_s\)^{\fS_{n+1}} = \bigoplus_{s= 0}^{2n}  \CH_{0}\(A^{n+1}_0\)^{\fS_{n+1}}_s.
\end{equation}
 Again, we recall that throughout this article the Chow groups are defined with rational coefficients. Since $q_*q^* :  \CH_0(K_{(n)}) \to \CH_0(K_{(n)})$ is the multiplication by $(n+1)!$~\cite[Example $1.7.6$]{Fulton}, hence bijective, the restriction to $\Ima(q^*)$  of the map
 $$q_* : \bigoplus_{s= 0}^{2n}  \CH_{0}\(A^{n+1}_0\)^{\fS_{n+1}}_s \to \CH_0(K_{(n)})$$
 is also bijective. Therefore we obtain the following decomposition
 \begin{equation}\label{eqn-BeauvilleDec}
 \CH_0(K_{(n)}) = \bigoplus_{s= 0}^{2n}   \CH_0(K_{(n)})_s,
 \end{equation}
where
$$  \CH_0(K_{(n)})_s \colonec  q_*\CH_{0}\(A^{n+1}_0\)^{\fS_{n+1}}_s = q_*\CH_{0}\(A^{n+1}_0\)_s.$$
The decomposition (\ref{eqn-BeauvilleDec}) of $\CH_0(K_{(n)})$ is called the \emph{induced Beauville decomposition}. Since the Hilbert-Chow morphism induces an isomorphism $\nu_* : \CH_0(K_{n}) \to \CH_0(K_{(n)})$ (see Remark~\ref{rem-quotsing}), this  also defines a decomposition on $\CH_0(K_{n})$.

\ssec{The vanishing of $\CH_0(K_{n})_{\text{odd}}$}\hfill

The goal of this subsection is to prove Theorem~\ref{thm-van}, which is a consequence of the following

\begin{thm}\label{thm-invtriv}
The involution of $A^{n + 1}_0$  acts trivially on $\CH_0(A^{n + 1}_0)^{\fS_{n+1}}$.
\end{thm}

\begin{proof}

%More generally, we will first construct an isomorphic correspondence described below at the level of  $\CH_0$ between $A^{n + 1}_0$ and the self-product of the quotient of $A$ by the involution action, then conclude that the involution action defined by that of $A$ acts trivially on $\CH_0(A^{n + 1}_0)^{\fS_{n+1}}$.

Instead of studying $\CH_0(A^{n + 1}_0)^{\fS_{n+1}}$, we will show that the involution of $A^{n}$ acts trivially on $\CH_0(A^{n})^{\fS_{n+1}}$, where the action of $\fS_{n+1}$ on $A^n$ is given by the action of $\fS_{n+1}$ on $A^{n + 1}_0$ \emph{via} the isomorphism 
$$A^{n + 1}_0 \eto A^n \ \ \ \ \ \(a_1,\ldots, a_{n}, -\sum_{j=1}^na_j\) \mapsto \(a_1,\ldots, a_{n}\).$$
Explicitly, if we identify  $\fS_{m}$ with the permutation group of $\bZ \cap [1,m]$ for each $m$ so that  $\fS_{n}$ is considered as a subgroup of $\fS_{n+1}$, then the action of $ \fS_{n+1}$ on $A^n$ is determined by the action of $\fS_{n} \subset \fS_{n+1}$ on $A^n$  permuting the factors, and by the action of any transposition $t_i$  exchanging $i$ and $n+1$ for $1 \le i \le n$ defined by
\begin{equation}\label{eqn-action}
t_i \cdot (z_1,\ldots,z_n) \colonec \(z_1,\ldots,z_{i-1}, -\sum_{j=1}^n z_j,z_{i+1},\ldots z_n\).
\end{equation}

For $1 \le j \le n$, let $p_j : A^n \to A$ denote the $j$-th projection. Since every zero-cycle $z \in \CH_0(A^n)$ can be decomposed as a sum of zero-cycles of the form $p_1^*z_1 \cdots p_n^*z_n$ where for each $j$, $z_j \in \CH_0(A)$ is either $\imath$-invariant or $\imath$-anti-invariant, Theorem~\ref{thm-invtriv} is a consequence of the following

\begin{lem}\label{lem-ratshit}
Let $z_1,\ldots,z_n \in \CH_0(A)$ be as above. If there exists $i$ such that $z_i$ is $\imath$-anti-invariant, then 
$$\sum_{\gs \in \fS_{n+1}}\gs^*\(\prod_{j=1}^n p_j^*z_j\) = 0.$$
\end{lem}

Before proving Lemma~\ref{lem-ratshit}, we note that

\begin{lem}\label{lem-seesawFM}
If $z_i$ is $\imath$-anti-invariant, then $\mu^*z_i = \sum_{j=1}^n p_j^*z_i$.
\end{lem}

\begin{proof}
Since  $z_i$ is $\imath$-anti-invariant, $z_i \in \CH_0(A)_1$
Let $L_i \colonec \cF(z_i) \in \Pic^0(A)$ be the Fourier-Mukai transform of $z_i$~\cite{BeauvilleTransfFM}. It follows from the Seesaw theorem~\cite[page 54]{MumfordAb} that $\mu^*L_i = \sum_{j=1}^n p_j^*L_i$. Applying Fourier-Mukai transform $\cF$ on the both sides of the preceding identity yields the result.
\end{proof}

\begin{proof}[Proof of Lemma~\ref{lem-ratshit}]
By definition, $p_i \circ t_i = -\mu$ and $p_j \circ t_i = p_j$ for all $i \ne j$. Together with Lemma~\ref{lem-seesawFM}, we see that
%\begin{lem}
\begin{equation}
\begin{split}
 t_i^*\(\prod_{j=1}^n p_j^*z_j\) = \prod_{j=1}^n t_i^*\(p_j^*z_j\) = (-\mu)^*z_i \cdot  \prod_{j\ne i}^n p_j^*z_j = - \prod_{j=1}^n p_j^*z_j.
% = p_1^*z_1 \cdots  p_{i-1}^*z_{i-1} \cdot -\mu^*z_i \cdot  p_{i+1}^*z_{i+1} \cdots  p_n^*z_n =  \prod_{j=1}^n p_j^*z_j.
\end{split}
\end{equation}
%- \prod_{j=1}^n p_j^*z_j.$$
%\end{lem}
%\begin{proof}
%\end{proof}

Hence
$$\sum_{\gs \in \fS_{n+1}}\gs^*\(\prod_{j=1}^n p_j^*z_j\) = \sum_{\gs \in \fA_{n}} \gs^*\(\prod_{j=1}^n p_j^*z_j+  t_i^*\(\prod_{j=1}^n p_j^*z_j\)\) = 0,$$
where $\fA_{n+1} < \fS_{n+1}$ stands for the alternating subgroup of $n+1$ elements.
\end{proof}

%\begin{thm}
%Let $\gS$ denote the quotient of $A$ by the involution action. Let $A^n$ be endowed with the $\fS_{n+1}$-action defined above. Then the quotient map $A^{n} \to \gS^{n}$ induces an isomorphism
%$$ \CH_0(A^{n})^{\fS_{n+1}} \eto  \CH_0(\gS^{n})^{\fS_{n+1}}.$$ 
%In particular,  for all $p \in \bZ$,
%$$\CH_0(K_{(n)})_{2p + 1} = 0.$$
 %\end{thm}
%\begin{proof}

%Let
%$$\CH_0(A^{n}) = \CH_0(A^{n})_+ \oplus \CH_0(A^{n})_-$$
%be the decomposition of $\CH_0(A^{n})$ such that the involution $(-1)_{A^{n}}$ acts as the identity map $\Id$ on $\CH_0(A^{n})_+$ and as $-\Id$ on $\CH_0(A^{n})_-$.

\end{proof}

\begin{proof}[Proof of Theorem~\ref{thm-van}]
By definition of the induced Beauville decomposition, the $\imath$-anti-invariant part of $\CH_0(K_{(n)})$ is identified with $\oplus_{p \in \bZ} \CH_0(K_{(n)})_{2p + 1}$. Thus by Theorem~\ref{thm-invtriv}, $\CH_0(K_{(n)})_{2p + 1}$ vanishes for all $p \in \bZ$.
\end{proof}

\begin{rem}

Let $\gS$ denote the quotient of $A$ under the involution $\imath$ and $\wt{\gS}$ the Kummer $K3$ surface defined by $A$. The action~\eqref{eqn-action} of $\fS_{n+1}$ on $A^n$ descends to an action on $\gS^n$. The quotient map $A \to \gS$ induces  a morphism 
$$ \CH_0(\gS^{(n)}) \to \CH_0(A^{(n)})$$
whose restriction to $ \CH_0(\gS^n)^{\fS_{n+1}}$ is isomorphic to the $\imath$-invariant part of $ \CH_0(A^n)^{\fS_{n+1}} \simeq \CH_0(K_n)$ (so the whole  $\CH_0(A^n)^{\fS_{n+1}} $ by Theorem~\ref{thm-invtriv}). However, the correspondence which defines the isomorphism $ \CH_0(\gS^n)^{\fS_{n+1}} \eto  \CH_0(K_n)$ does not come from a morphism $f : \wt{\gS}^{[n]} \to K_n$, which is why we cannot compare the rational orbit filtration of $\CH_0(\wt{\gS}^{[n]})$ and that of $\CH_0(K_n)$.

\end{rem}

We finish this section by stating the Chow-Künneth decomposition for zero-cycles on $K_n$, which is a direct consequence of the existence of Chow-Künneth decomposition for abelian varieties.

\begin{pro}\label{pro-ChowKun}
There exist $\pi_1,\ldots,\pi_n \in \CH^{2n}\(K_{n} \times K_{n}\)$ such that for all $1 \le j \le n$, 
\begin{enumerate}[i)]
\item ${\pi_j}_*$ acts as the projection $\CH_0(K_{n}) \to \CH_0(K_{n})_{2j}$ with respect to decomposition~\ref{eqn-decomp};
\item $\pi_j^*$ acts as the identity map on $H^{0}(K_n,\gO_{K_n}^l)$ if $l = 2j$ and as $0$ otherwise.
\end{enumerate}
\end{pro}

\begin{proof}
As we recalled at the beginning of this section, there exist $\Delta_1,\ldots,\Delta_{4n} \in \CH^{2n}(A^{n+1}_0 \times A^{n+1}_0)$ such that ${\Delta_j}_*$ acts as the projection $\CH_0(A^{n+1}_0) \to \CH_0(A^{n+1}_0)_{4n-j}$ with respect to the Beauville decomposition and that $[\Delta] = \sum_{j=1}^{4n}[\Delta_j]$ in $H^{2n}(A^{n+1}_0,\bC)$ is the Künneth decomposition. If $\pi_1,\ldots,\pi_{n} \in \CH^{2n}(K_{n} \times K_{n})$ verify 
$$(\nu \times \nu)_*\pi_j = (q_{n+1} \times q_{n+1})_*\Delta_{4n-2j}$$ 
where we recall that $\nu : K_n \to K_{(n)}$ is the Hilbert-Chow map and $q_{n+1} : A^{n+1}_0 \to K_{(n)}$ is the quotient map, then $\pi_1,\ldots,\pi_{n} $ satisfy the properties listed in Proposition~\ref{pro-ChowKun}.

\end{proof}

\section{The rational orbit filtration and the induced Beauville decomposition coincide}\label{sec-coin}

The last part of this paper is devoted to the comparison between the rational orbit filtration and the induced Beauville decomposition of a generalized Kummer variety.

\ssec{Proof of Theorem~\ref{thm-coin}} \hfill

Recall that we want to prove $S_p \CH_0(K_n) = \CH_0(K_n)_{\le 2n- 2p}$. We first prove one inclusion

\begin{lem}\label{lem-inclu}
For all $0 \le p \le n$, 
$$S_p \CH_0(K_{n}) \subset \CH_0(K_{n})_{\le 2n- 2p}.$$
\end{lem}
\begin{proof}

By Corollary~\ref{cor-main}, it suffices to show that for all $0 \le p \le n$ and $k \in \bZ_{>0}$,
$$\Ima \(\CH_0(V_{p,k}) \to \CH_0(K_{(n)})\) \subset \CH_0(K_{n})_{\le 2n- 2p}.$$
Let $z \in V_{p,k}$. If $p = n$, then $z \sim_\rat (n+1) \cdot \{0\}$  in $K_{(n)}$. Since $\{(0,\ldots,0)\} \in \CH_0(A_0^{n+1})_0$, we see that 
$$\{z\} = q_*\{(0,\ldots,0)\} \in  \CH_0(K_{(n)})_0$$
where $q : A^{n+1}_0 \to K_{(n)}$ stands for the quotient map.

 Now assume that $p < n$. By Lemma~\ref{lem-rep}, $z$ is rationally equivalent in $K_{(n)}$ to some element of the form 
$$p \cdot \{a\} +\{a+ c\}+\{a+ c'\} + \sum_{j = p+3}^{n+1}\{a+ a_{j}\}$$ 
for some $c,c' \in C_k$. Hence by Lemma~\ref{lem-facile} below, in $ \CH_0(K_{(n)})$ we have
$$(n+1)^2\cdot \{z\} = q_*\imath^*\tau_*\(\prod_{j=1}^p\pi_j^*\{0\}   \cdot \pi_{p+1}^*\{c\}\cdot \pi_{p+2}^*\{c'\}\cdot  \prod_{j=p+3}^{n+1}\pi_{j}^*\{a_{j}\} \)$$ where $\pi_j : A^{n+2}  \to A$ denotes the $j$-th projection, $\imath : A^{n+1}_0 \hookrightarrow A^{n+1}$ the inclusion map, and $\tau : A^{n+2} \to A^{n+1}$ was defined as~(\ref{def-tau}). Since $\{0\} \in \CH_0(A)_0$ and $\{c\},\{c'\} \in \CH_0(A)_{\le 1}$, we conclude that $\{z\} \in \CH_0(K_{(n)})_{\le 2n- 2p}$ by~\cite[Proposition $2$]{BeauvilleDecAb}. 
\end{proof}

 %Let $\pr_1 : A^{n+1} \times A \to A^{n+1}$ be the first projection. 
 Using the same notations as in the proof of  Lemma~\ref{lem-inclu}, we have the following easy
 \begin{lem}\label{lem-facile}
  Let $z \in \CH_0(A^{n+1})$ be a zero-cycle supported on $A_0^{n+1}$, then
\begin{equation}\label{eqn-pp}
 (n+1)^2\cdot z = \imath^*\(\tau_*(z \times A)\) .
 \end{equation}
 \end{lem}
\begin{proof}

Let $z \in A_0^{n+1} \subset A^{n+1}$, then
$$A_0^{n+1} \cap \tau(z,A) = \{\tau_a(z) \mid (n+1) \cdot a = 0\}.$$
It follows that as zero-cycles,
$$\imath^*\(\tau_*(z \times A)\) = \sum_{a \in A[n+1]} \tau^*_az = (n+1)^2\cdot z$$
where the last equality follows from~\cite[Theorem $1$]{Huibregtse}.
\end{proof}

\begin{pro}\label{pro-suppBD}
For all $0 \le p \le n$,
$$ \CH_0(K_{(n)})_{\le 2n- 2p + 1} \subset \Ima\( \CH_0(V_{p,1}) \to \CH_0(K_{(n)}) \).$$
\end{pro}

 \begin{proof}

Given $z \in \CH_0(A_0^{n+1})_{\le 2n- 2p + 1}$, so $\imath_*z \in \CH_0(A^{n+1})_{\le 2n- 2p + 1}$, then by~\cite[Proposition $4$]{BeauvilleDecAb} applying to the symmetric ample divisor $\sum_{j=1}^{n+1}\pi_j^*C$, $\imath_*z$ is supported on 
$$\bigcup_{\substack{j,l \ge 0, j+l \le n+1 \\ 2j + l = 2p + 1}} W_{j,l}.$$
where $W_{j,l}$ denotes the orbit of  $\{0\}^j \times C^l \times A^{n+1-j-l} \subset A^{n+1}$ under the permutation of factors.  Let $z'$ be a zero-cycle supported on $W_{j,l}$. Since $q_*z'$ is proportional to $q_*\(\imath^*\tau_*(z' \times A)\)$ in $\CH_0(K_{(n)})$ by Lemma~\ref{lem-facile}, it suffices to show that the later is supported on $V_{p,1}$ to finish the proof.

%Let $z \in \CH_0(K_{(n)})_{\le 2n- 2p + 1}$
%Now let $z \in \CH_0(A_0^{n+1})_{\le 2n- 2p}$. Note that
%\begin{equation}
%[n+1]_*\imath^*  \tau_* p^* \imath_* z = (n+1)^2\cdot [n+1]_*z.
%\end{equation}

%\begin{lem}\label{lem-projiphi}
%$\imath^* \circ \tau_* : \CH_2(A^{n+1} \times A) \to \CH_0(A_0^{n+1})$ is surjective.
%\end{lem}
%\begin{proof}
%Let $p : A^{n+1} \times A \to A^{n+1}$ denote the first projection. Lemma~\ref{lem-projiphi} follows from the following identity for every $z \in \CH_0(A_0^{n+1})$:
%$$[n+1]_*\imath^*  \tau_* p^* \imath_* z = (n+1)^2\cdot [n+1]_*z.$$
%\end{proof}

%Let $z \in \CH_0(K_{(n)})_{\le 2n- 2p + 1}$, then there exists $z' \in \CH_2(A^{n+1} \times A)_{\le 2n- 2p + 1}$ such that $q_*\imath^*\tau_*z' = z$. 

 By definition of the $V_{p,1}$'s, if  $l > 1$ then $q_*\(\imath^*\tau_*(z' \times A)\)$ is supported on $V_{p,1}$. Assume that $l= 1$ then there exist $c \in C$ and $a,a_{p+2},\ldots,a_{n+1} \in A$ such that $q_*\(\imath^*\tau_*(z' \times A)\)$ is rationally equivalent to a sum (as $0$-cycles) of elements in $K_{(n)}$ of the form $z'' \colonec p\cdot \{a\} + \{a+c\} + \sum_{m = p+2}^{n+1}\{a_m\}$.

\begin{lem}\label{lem-re}
For any $k \in \bZ_{>0}$ and $c' \in C_k$,
$$p\cdot \{-c'\} + \{p\cdot c'\} = (p+1)\cdot \{0\}$$
in $\CH_0(K_{(p)})$.
\end{lem}
\begin{proof} 
We prove Lemma~\ref{lem-re} by induction on $p \ge 0$. The case $p=0$ is obviously true. Lemma~\ref{lem-re} holds also for $p=1$ since $c' \in C_k$. Suppose that $p \ge 2$, since $2\cdot \{0\} = \{(p-1)c'\} + \{-(p-1)c'\}$  in $\CH_0(K_{(1)})$, we see that in $\CH_0(K_{(p)})$,
\begin{equation*}
\begin{split}
p\cdot \{-c'\} + \{p\cdot c'\} &=  (p-2)\cdot \{-c'\} + \tau_{c'}^*\(\{(p-1)c'\} + \{-(p-1)c'\}\) + \{p\cdot c'\} \\
&=  (p-2)\cdot \{-c'\} + \{(p-2)c'\} + \{-pc'\} + \{p\cdot c'\}=   (p+1)\cdot \{0\}
\end{split}
\end{equation*}
where the last equality results from induction hypothesis and the fact that $\{-pc'\} + \{p\cdot c'\} = 2\cdot \{0\}$.
\end{proof}

%Then by Lemma~\ref{lem-rep}, $z''$ is rationally equivalent to $(p + 1)\cdot \{a + c'\}+ \sum_{m = p+2}^{n+1}\{a_m\}$, hence $z'' \in S_p\CH_0(K_{(n)})$ by Lemma~\ref{lem-rep}. 
Back to the proof of Proposition~\ref{pro-suppBD}, let $c' \in C_{p + 1}$ such that $(p + 1)\cdot c' = c$. By Lemma~\ref{lem-re}, we see that  $z''$ is rationally equivalent to $(p+1)\cdot \{a +c'\} +  \sum_{m = p+2}^{n+1}\{a_m\}$. If $p=n$, then $a+c'$ is a torsion point so $z'' \sim_\rat (n+1)\cdot \{0\}$ by~\cite[Theorem $1$]{Huibregtse}. If $p < n$, since there exist $a' \in A$ and $c_1,c_2 \in C$ such that $a' + c_1 =a +c' $ and $a' + c_2 = a_{p+2}$,  $z''$ is rationally equivalent to $(p+1)\cdot \{a' +c_1\} + \{a' + c_2\}  + \sum_{m = p+3}^{n+1}\{a_m\}$. In either case, we conclude that $z''$ is supported on $V_{p,1}$.

%, then$$z' \in \left\{         p\cdot \{a\} + \{a+c\} + \sum_{m = p+1}^n\{a_m\} \  \big{|} \       \right\}$$

%Let $z = (0,\ldots,0 , a_{p+1},\ldots,a_{n+1}) \in A^{n+1}_0$, then $[z] = \imath^*\(\pr_1^*[0] \cdots \pr_{p}^*[0]  \cdot \pr_{p+1}^*[a_{p+1}] \cdots \pr_{n}^*[a_{n}]\)$ in $\CH_0(A_0^{n+1})$ where $\pr_j : A^{n+1} \to A$ is the $j$-th projection. Since $[0] \in \CH_0(A)_0$,  $q_*[z] \in \CH_0(K_{(n)})_{\le 2n- 2p}$ by~\cite[Proposition $2$]{BeauvilleDecAb}. 
\end{proof}

%\begin{thm}\label{main1}%\item  The quotient map $q$ preserves the Beauville decomposition of $\CH^{2n}(A^{(n+1)}_0)$. Namely,
%For all $0 \le p \le n$, 
%$$ \CH_0(K_{(n)})_{\le 2n- 2p} = \CH_0(K_{(n)})_{\le 2n- 2p + 1}  = S_p \CH_0(K_{(n)}).$$
%\end{thm}

\begin{proof}[Proof of Theorem~\ref{thm-coin}]
Since $\nu^*\Ima\( \CH_0(V_{p,1}) \to \CH_0(K_{(n)}) \) \subset S_p\CH_0(K_n)$ by Proposition~\ref{pro-Ci}, it follows from Lemma~\ref{lem-inclu} and Proposition~\ref{pro-suppBD} that
\begin{equation}\label{eqn-chaininclu}
S_p \CH_0(K_{n}) \subset \CH_0(K_{n})_{\le 2n- 2p} \subset \CH_0(K_{n})_{\le 2n- 2p + 1} \subset  \nu^*\Ima\( \CH_0(V_{p,1}) \to \CH_0(K_{(n)}) \) \subset S_p\CH_0(K_n).
\end{equation}
\end{proof}

\ssec{Final remarks}

 \hfill
\begin{enumerate}[i)]
\item
Note that the chain of inclusions~(\ref{eqn-chaininclu}) gives a second proof of Theorem~\ref{thm-van}. If we are only interested in proving Theorem~\ref{thm-coin}, instead of proving Proposition~\ref{pro-suppBD} we could have just shown that $ \CH_0(K_{(n)})_{\le 2n- 2p} \subset \Ima\( \CH_0(V_{p,1}) \to \CH_0(K_{(n)}) \)$ and used Theorem~\ref{thm-van} to conclude.
\item
Combining Theorem~\ref{thm-coin} and Proposition~\ref{pro-ChowKun}, we obtain a positive answer of~\cite[Conjecture 0.8]{VoisinHKcoisotp} for generalized Kummer varieties.
\item
Finally we note that  the chain of inclusions~(\ref{eqn-chaininclu}) also implies that $S_p\CH_0(K_{n})$ is supported on a subvariety of codimension $p$, while in Corollary~\ref{cor-main},  $S_p\CH_0(K_{n})$ is only proved to be supported on a \emph{countable union} of subvarieties of codimension $p$:
\end{enumerate}

\begin{cor}
For all $0 \le p \le n$,
$$\Ima\( \CH_0(V_{p,1}) \to \CH_0(K_{(n)}) \) = \nu_*S_p\CH_0(K_{n}).$$
\end{cor}

%Finally, by the works of Beauville, Deninger-Murre, and Shermenev

\section*{Acknowledgement}
I would like to thank my thesis advisor Claire Voisin for the discussions we had and her help concerning the presentation of this work.

\bibliographystyle{plain}
\bibliography{CH0genKummer}

\begin{thebibliography}{10}

\bibitem{BeauvilleTransfFM}
A.~Beauville.
\newblock Quelques remarques sur la transformation de {F}ourier dans l'anneau
  de {C}how d'une variété abélienne.
\newblock In {\em Algebraic geometry (Tokyo/Kyoto 1982)}, volume 1016 of {\em
  Lect. Notes Math.}, pages 238--260. Springer, 1983.

\bibitem{BeauvilleDecAb}
A.~Beauville.
\newblock Sur l'anneau de {C}how d'une variété abélienne.
\newblock {\em Math. Ann.}, 273:647--651, 1986.

\bibitem{BeauvilleSplitBB}
A.~Beauville.
\newblock On the splitting of the {B}loch-{B}eilinson filtration.
\newblock In {\em Algebraic cycles and motives (vol. $2$)}, volume 344 of {\em
  London Math. Soc. Lecture Notes}, pages 38--53. Cambridge University Press,
  2007.

\bibitem{BV}
A.~Beauville and C.~Voisin.
\newblock On the chow ring of a ${K}3$ surface.
\newblock {\em J. Algebraic Geom.}, 13(3):417--426, 2004.

\bibitem{Blochab}
S.~Bloch.
\newblock Some elementary theorems about algebraic cycles on abelian varieties.
\newblock {\em Invent. math.}, 37:215--228, 1976.

\bibitem{Briancon}
J.~Briançon.
\newblock Description de ${H}ilb^nc\{x,y\}$.
\newblock {\em Invent. Math.}, 41:45--89, 1977.

\bibitem{DeningerMurre}
C.~Deninger and J.~Murre.
\newblock Motivic decomposition of abelian schemes and the {F}ourier transform.
\newblock {\em J. reine angew. Math.}, 422:201--219, 1991.

\bibitem{FuBVKummer}
L.~Fu.
\newblock Beauville-{V}oisin conjecture for generalized {K}ummer varieties.
\newblock {\em International Mathematics Research Notices}, 2014.

\bibitem{Fulton}
W.~Fulton.
\newblock {\em Intersection theory}, volume~2 of {\em {Ergebnisse der
  Mathematik und ihrer Grenzgebiete. 3. Folge.}}
\newblock {Springer-Verlag}, Berlin, {Second} edition, 1998.

\bibitem{Huibregtse}
M.E. Huibregtse.
\newblock Translation by torsion points and rational equivalence of 0-cycles on
  abelian varieties.
\newblock {\em Rocky Mountain J. Math.}, 18(1):147--156, 1988.

\bibitem{HuybCCC}
D.~Huybrechts.
\newblock Curves and cycles on {K}3 surfaces.
\newblock {\em Algebraic Geometry}, 1:69--106, 2014.

\bibitem{JannsenFBBM}
U.~Jannsen.
\newblock Motivic sheaves and filtrations on {C}how groups.
\newblock In {\em Motives (Seattle, WA, 1991)}, volume~55 of {\em Proceedings
  of Symposia in Pure Mathematics}, pages 245--302. American Mathematical
  Society, 1994.

\bibitem{Kunnemann}
K.~Künnemann.
\newblock On the {C}how motive of an abelian scheme.
\newblock In {\em Motives (Seattle, WA, 1991)}, volume~55 of {\em Proceedings
  of Symposia in Pure Mathematics}, pages 189--205. American Mathematical
  Society, 1994.

\bibitem{MumfordAb}
D.~Mumford.
\newblock {\em Abelian varieties}, volume~5 of {\em Tata Institute of
  Fundamental Research Studies in Mathematics}.
\newblock Oxford University Press, 1970.

\bibitem{MurreFiltrBBM}
J.~P. Murre.
\newblock On a conjectural filtration on the {C}how groups of an algebraic
  variety. {P}art $i$: {T} general conjectures and some examples.
\newblock {\em Indag. Mathem. New Series}, 4(2):177--188, 1993.

\bibitem{RiessBconjLF}
U.~Rie\ss.
\newblock On the {B}eauville conjecture.
\newblock {\em arXiv:1409.3484}, 2014.

\bibitem{ShenVialFM}
M.~Shen and C.~Vial.
\newblock The fourier transform for certain hyperkaehler fourfolds.
\newblock {\em arXiv:1309.5965}, 2013.

\bibitem{Shermenev}
A.~M. Shermenev.
\newblock The motive of an abelian variety.
\newblock {\em Funct. Anal.}, 8:55--61, 1974.

\bibitem{VoisinII}
C.~Voisin.
\newblock {\em Hodge Theory and Complex Algebraic Geometry II}, volume~77 of
  {\em Cambridge Studies in Advanced Mathematics}.
\newblock Cambridge University Press, 2003.

\bibitem{VoisinBVconjHilbK3}
C.~Voisin.
\newblock On the chow ring of certain algebraic hyper-{K}ähler manifolds.
\newblock {\em Pure Appl. Math. Q.}, 4(3):613--649, 2008.

\bibitem{Voisin0cycleK3}
C.~Voisin.
\newblock Rational equivalence of $0$-cycles on ${K}3$ surfaces and conjectures
  of {H}uybrechts and {O}'{G}rady.
\newblock {\em To appear in ''Recent Advances in Algebraic Geometry, a
  conference in honor of Rob Lazarsfeld's 60th birthday''}, 2012.

\bibitem{VoisinHKcoisotp}
C.~Voisin.
\newblock Remarks and questions on coisotropic subvarieties and $0$-cycles of
  hyper-{K}ähler varieties.
\newblock {\em arXiv:1501.02984}, 2015.

\end{thebibliography}

\end{document}